\documentclass[11pt,twoside, reqno]{article}

\usepackage{amsfonts}
\usepackage{amssymb}
\usepackage{amsmath}
\usepackage{amsthm}
\usepackage{latexsym}
\usepackage{mathrsfs}
\usepackage{bbm}

\usepackage{txfonts}
\usepackage{enumerate}
\usepackage{hyperref}
\usepackage{graphicx}
\usepackage{authblk}

\usepackage{color}
\usepackage{pgf,tikz}
\usepackage{mathrsfs}
\usetikzlibrary{arrows}
\allowdisplaybreaks

\usepackage{indentfirst}

\def\fz{\infty}

\def\r{\right}
\def\lf{\left}

\newcommand\SM{\mathscr M}
\def\az{\alpha}
\def\lz{\lambda}

\def\sz{\sigma}

\def\ls{\lesssim}

\newcommand\newdot{{\kern.8pt\cdot\kern.8pt}}
\newcommand\smallbullet{{\vcenter{\hbox{\tiny$\bullet$}}}}

\def\dsum{\displaystyle\sum}

\def\dfrac{\displaystyle\frac}
\def\hs{\hspace{0.3cm}}

\def\e{{\rm e}}
\def\mathpal#1{\mathop{\mathchoice{\text{\rm #1}}%
   {\text{\rm #1}}{\text{\rm #1}}%
   {\text{\rm #1}}}\nolimits}
\newcommand\Hess{\mathpal{Hess}}
\newcommand\Ric{\mathpal{Ric}}

\newcommand\vol{\mathpal{vol}}

\newcommand\tr{\mathpal{tr}}

\newcommand\1{\mathbbm{1}}

\newcommand\R{\mathbb{R}}
\newcommand\N{\mathbb{N}}

\newtheorem{theorem}{Theorem}[section]
\newtheorem{lemma}[theorem]{Lemma}
\newtheorem{corollary}[theorem]{Corollary}

\theoremstyle{definition}

\newtheorem{definition}[theorem]{Definition}
\newtheorem{assumption}[theorem]{Assumption}

\def\supp{{\mathop\mathrm{\,supp\,}}}

\def\d{\mathrm{d}}
\numberwithin{equation}{section}

\begin{document}

\title{\vskip-2.1cm\bf\Large Covariant Riesz transform on differential forms for  $1<p\leq2$
  \footnotetext{\hspace{-0.35cm} 2010 {\it Mathematics Subject
      Classification}. Primary: 35K08; Secondary: 58J65, 35J10, 47G40.
    \endgraf {\it Keywords and phrases}.  Covariant Riesz transform,  Heat kernel, Bochner formula,
    Calder\'{o}n-Zygmund inequalities, 
    Hardy-Littlewood maximal function, 
    Kato's inequality.  \endgraf This work has been supported by
    the National Natural Science Foundation of China (Grant
    No. 11831014,11921001) and  Natural Science Foundation of Zhejiang Provincial (Grant No.LGJ22A010001). The second author acknowledges support by the Fonds
    National de la Recherche Luxembourg (project GEOMREV
    O14/7628746).}}

\author[1]{Li-Juan Cheng} \author[2]{Anton Thalmaier} \author[3,4]{Feng-Yu Wang}

\affil[1]{\small School of Mathematics, Hangzhou Normal
  University,\par
  Hangzhou 311121, People's Republic of China\par
  \texttt{chenglj@zjut.edu.cn}\vspace{1em}}
  
  \affil[2]{\small Department of Mathematics, University of Luxembourg,
  Maison du Nombre,\par
  L-4364 Esch-sur-Alzette, Luxembourg\par
  \texttt{anton.thalmaier@uni.lu}\vspace{1em}}
  
  \affil[3]{\small Center for Applied Mathematics, Tianjin
      University,\par Tianjin 300072, People's Republic of China}
    \affil[4]{\small Department of Mathematics, Swansea University,
      Bay Campus,\par
      Swansea SA1 8EN, United Kingdom\par
      \texttt{wangfy@bnu.edu.cn}}

\date{\today}
\maketitle

\begin{abstract} {\noindent In this paper, we study $L^p$-boundedness ($1<p\leq 2$) of the covariant Riesz transform on differential forms for a class of non-compact weighted Riemannian manifolds without assuming conditions on derivatives of curvature. We present in particular a local version of $L^p$-boundedness of
 Riesz transforms under two natural conditions, namely the curvature-dimension condition, and a
lower bound on the Weitzenb\"{o}ck curvature endomorphism. As an application, the  Calder\'on-Zygmund inequality for $1< p\leq  2$ on weighted manifolds is derived under the curvature-dimension condition as hypothesis.}
\end{abstract}

\tableofcontents

\section{Introduction\label{s1}}
Let $(M,g)$ be a  complete geodesically connected $m$-dimensional
Riemannian manifold, $\nabla$ the Levi-Civita covariant derivative, and 
$\Delta$ the Laplace-Beltrami operator understood as a self-adjoint positive operator on $L^2(M)$. The Riesz transform $\nabla \Delta^{-1 / 2} f$,  introduced by Strichartz \cite{Str83} on Euclidean space, has been investigated in many subsequent papers, see e.g., \cite{Chen92,Barkry85I,Barkry85II,Bakry87} and the references therein, and has been further extended to Riemannian manifolds, e.g.~\cite{PTTS-2004,TX-99,CD-01,CD-03,Li91,ThW-04}.  Since the Riesz transform is bounded in $L^{2}(M)$, by the interpolation theorem, the weak $(1,1)$ property already implies $L^{p}(M)$-boundedness for $p \in(1,2]$.

When it comes to Riemannian vector bundles, boundedness in $L^p$ of  the Riesz transforms  $\d^{(k)} (\Delta^{(k)}+\sigma)^{-1/2}$ and $\delta^{(k-1)}(\Delta^{(k)}+\sigma)^{-1/2}$ has been well considered, see  \cite{Bakry87}, where $\Delta^{(k)}$ is the usual Hodge Laplacian acting on $k$-forms, $\d^{(k)}$ the exterior differential on $k$-forms and $\delta^{(k)}$ the $L^2$-adjoint of $\d^{(k)}$. Let $\nabla$ be the Levi-Civita covariant derivative.
In this paper, we aim to study covariant Riesz transforms $\nabla (\Delta^{(k)}+\sigma)^{-1/2}$ on Riemannian vector bundles for $p \in(1,2]$ which poses comparably more difficulties to deal with. This question has already been addressed by the second and third named author in \cite{ThW-04}, where the authors adopted the method of Coulhon and Duong \cite{TX-99} relying on the doubling volume property, Li-Yau type heat kernel upper bounds and derivative estimates of the heat kernel. Note that the approach in \cite{ThW-04} is of stochastic nature and derivative estimates for the heat kernel are deduced from derivative formulae for semigroups on vector bundles, by means of the methodology of Driver and the second named author \cite{DT01}. In \cite{DT01}   estimates of certain functionals of Brownian motion with respect to the Wiener measure are required; nevertheless  pointwise estimate for the heat kernel
$\e^{-t{\Delta}^{(k)}}(x,y)$  and the derivative estimate of heat kernel $\nabla \e^{-t{\Delta}^{(k)}}(x,y)$ can be obtained from such general derivative formulas, where ${\Delta}^{(k)}$ is the 
unique self-adjoint realization of the Hodge-de Rham Laplacian acting
on $k$-forms under explicit curvature condition, see \cite{BDG-21}. By following a similar approach, Baumgarth, Devyver and G\"{u}neysu \cite{BDG-21}
studied the covariant Riesz transform on $j$-forms, removing the doubling volume property and using uniformly boundedness conditions of the curvature and the derivative  of the curvature on differential forms. These results are further  used to establish  Calder\'{o}n-Zygmund inequalities for $1<p\leq  2$, where
for $\varphi\in C_c^{\infty}(M)$,  the set of smooth
  functions of compact support, if there exist positive constants $C_1$ and $C_2$ such that
\begin{align*}
\|\,|\Hess (\varphi)|\,\|_{p}\leq  C_1 \|\varphi\|_p+C_2  \|\Delta \varphi\|_p
\end{align*}
then the  Calder\'{o}n-Zygmund inequalities holds. Note that the argument in \cite{BDG-21} for establishing  Calder\'{o}n-Zygmund inequalities is from G\"uneysu and Pigola's paper  \cite{GP-15}, where  the  Calder\'{o}n-Zygmund inequalities are equivalent to the $L^p$ boundedness of the operator $\Hess (\Delta +\sigma)^{-1}$ for some constant
$\sigma>0$ and this operator further can be rewritten as
\begin{align*}
\nabla (\Delta^{(1)}+\sigma)^{-1/2} \circ {\rm d}(\Delta+\sigma)^{-1/2}. 
\end{align*}
Thus  
to investigate whether  these $L^p$-Calder\'{o}n-Zygmund inequalities hold  are  reduced to the study of conditions for boundedness
of the classical Riesz transform ${\rm d}(\Delta_{\mu}+\sigma)^{-1/2}$ on
functions and boundedness of the covariant Riesz transform
$\nabla (\Delta_{\mu}^{(1)}+\sigma)^{-1/2}$ on one-forms in $L^p$-sense. Therefore, in \cite{BDG-21}    combining this argument with the result in \cite{ThW-04} yields  that the
$L^p$-Calder\'{o}n-Zygmund inequalities hold for $1<p\leq 2$ if
$$\|R\|_{\infty}<\infty \quad  \mbox{and}   \quad  \|\nabla R\|_{\infty}<\infty,$$
where $R$ is the curvature tensor.

On the other hand, very recently, Cao, Cheng and Thalmaier \cite{CCT} 
   established the $L^p$-Calder\'{o}n-Zygmund inequality for $1< p<2$ by only using the natural
    assumption of a lower Ricci curvature bound. Since the conditions for the $L^p$ boundedness
of the classical Riesz transform ${\rm d}(\Delta_{\mu}+\sigma)^{-1/2}$ for $1< p\leq 2$ are quite weak, it arises us to wondering whether the condition for the $L^p$ boundedness of  the covariant Riesz transform on one-forms in \cite{BDG-21} is too strong.  This work is devoted to investigating this problem
on the $L^p$ boundedness of covariant Riesz transform for $1<p\leq 2$ again.
     
 As explained in \cite{ThW-04}, it is difficult to follow the corresponding argument in  \cite{TX-99} directly concerning derivatives of heat kernel on vector bundle, since the heat kernel $\e^{-t \Delta^{(k)}}(x,y)$ is  linear operator on a vector bundle $E\rightarrow M$ from $E_y$ to $E_x$. 
 In this paper, we aim to overcome this difficulty by using a different approach, that is the Weitzenb\"{o}ck formula and the Gaussian type estimates for some Schr\"{o}dinger heat kernels on manifolds.

Before moving on, let us first introduce some basic notations. Consider a weighted Laplacian $\Delta+\nabla h$ with $h\in C^2_b(M)$. In this paper, we  study the covariant Riesz transform relative to the weighted volume measure $\mu(\mathrm{d} x)=\mathrm{e}^{h(x)} \vol( \mathrm{d} x)$ where $\vol$ is the
Riemannian volume measure on~$M$. We write $\Delta_{\mu}:=\Delta+\nabla h$ where $\Delta_{\mu}$  is understood as a self-adjoint positive operator on  $L^2(\mu)$.
Let $\rho(x,y)$ be the geodesic distance of $x$ and $y$ and $B(x,r)$ the open ball
centered at $x$ of radius $r$. Given a smooth vector bundle $E \rightarrow M$ carrying a canonically given metric and a canonically given covariant derivative, we denote its fiberwise metric by $(\newdot,\newdot)_g$,  the fiberwise norm by $|\cdot|=\sqrt{(\newdot, \newdot)_g}$  and the smooth
sections by $\Gamma_{C^{\infty}}(M, E)$. 
We denote its covariant derivative by
$$
\nabla: \Gamma_{C^{\infty}}(M,\, E) \rightarrow \Gamma_{C^{\infty}}\left(M,\, T^{*} M \otimes E\right).
$$
The Banach space $\Gamma_{L^{p}}(M,\, E)$ consists of equivalent classes of Borel sections $\psi$ of $E \rightarrow M$ such that
$$
\|\psi\|_{p}\equiv\|\psi\|_{L^{p}}:=\|\,|\psi|\,\|_{L^{p}}<\infty
$$
where $\|\,|\psi| \,\|_{L^p}$ denotes the norm of the function $|\psi|$ with respect to $L^{p}(\mu)$. Then $\Gamma_{L^{2}}(E)$ canonically becomes a Hilbert space with the scalar product
$$
\left\langle\psi_{1}, \psi_{2}\right\rangle:=\left\langle\psi_{1}, \psi_{2}\right\rangle_{L^{2}}=\int\left(\psi_{1}, \psi_{2}\right)_g \, \d \mu .
$$

Consider the spaces
$$
\Omega^{k}=\Gamma_{C^{\infty}}\big(M,\, \Lambda^{k} T^{*} M\big)\quad \mbox{and} \quad \Omega^{k}_c=\Gamma_{C^{\infty}_c}\big(M, \, \Lambda^{k} T^{*} M\big)\quad(0 \leq  k \leq  m)
$$
of smooth differential $k$-forms, respectively compactly supported smooth $k$-forms, and denote
the space of smooth $k$-forms in $L^p$ by
\begin{align*}
\Omega_{L^p}^k:=\Gamma_{C^{\infty}\cap L^p(\mu)}\big(M,\Lambda^{k} T^{*} M\big).
\end{align*} 
In terms of the exterior differential $\mathrm{d}^{(k)}: \Omega^{k} \rightarrow \Omega^{k+1}$
on $\Omega^{k}$ and $\delta_{\mu}^{(k+1)}$ the $L^{2}(\mu)$-adjoint of $\d^{(k)}$, i.e., \begin{align*}
\big\langle\delta_{\mu}^{(k+1)} a, b\big\rangle:=\mu\big((\delta_{\mu}^{(k+1)} a, b)_g\big)=\mu\big((a, \mathrm{d}^{(k)}  b)_g\big):=\langle a, \mathrm{d}^{(k)} b\rangle
\end{align*}
for $a\in\Omega^{k+1}$ and $b\in\Omega^{k}$, 
the weighted Hodge Laplacians acting on $0$-, respectively $k$-forms, are given by
\begin{align}
&\Delta^{(0)}_{\mu}:=\delta_{\mu}^{(1)} \mathrm{d}^{(0)}: C^{\infty}(M) \rightarrow C^{\infty}(M),  \notag \\
&\Delta_{\mu}^{(k)}:=\delta_{\mu}^{(k+1)} \mathrm{d}^{(k)}+\mathrm{d}^{(k-1)} \delta_{\mu}^{(k)}: \Omega^k \rightarrow \Omega^k.  \label{Hodge-Laplacian}
\end{align}
Obviously, the canonical commutation rules hold:
$$
\mathrm{d}^{(k-1)} \Delta_{\mu}^{(k-1)}=\Delta_{\mu}^{(k)} \mathrm{d}^{(k-1)}.
$$
To simplify the notation, we write $\Delta_{\mu}=\Delta_{\mu}^{(0)}$ and $\d=\d^{(0)}$.



To this end, we first give some assumptions where we start with notions related to the curvature. Letting $\nabla_{\mu}^{*}$ be the $L^2(\mu)$-adjoint
of $\nabla$ and  $\square_{\mu}=\nabla_{\mu}^{*}\nabla$ the weighted Bochner Laplacian, it is easy to check that 
\begin{align}\label{Bochner-Laplacian}
\square_{\mu}=-\operatorname{tr} \nabla^{2}-\nabla_{\nabla h},
\end{align}
where  $$
\tr\nabla^{2} \eta(\smallbullet):=\sum_i \nabla^{2} \eta\left(\smallbullet, e_{i}, e_{i}\right)
$$
with $\eta$ being a differential $k$-form and $\left(e_{i}\right)$ a local orthonormal frame.  Note that  by definition $\nabla^{2} \eta$ is a tensor of order $(0, k+2)$ and 
$\tr\nabla^{2} \eta$  is independent of the choice of local orthonormal frame $\left(e_{i}\right)$.
  The Weitzenb\"{o}ck formula gives the relationship  between $\square_{\mu}$ and the Hodge Laplacian $\Delta_{\mu}^{(\cdot)}$:
 for any differential $k$-forms $\eta\in\Omega^{k}$, we have
\begin{align}\label{Eq:Weitzenboeck}
\Delta_{\mu}^{(k)} \eta=\square_{\mu}\eta+\mathscr{R}^{(k)}(\eta)- (\Hess h)^{(k)}(\eta),
\end{align}
where in explicit terms the {\it Weitzenb\"{o}ck curvature endomorphism } $\mathscr{R}^{(k)}-(\Hess  h)^{(k)}$ is given by
$$
\big(\mathscr{R}^{(k)}-(\Hess  h)^{(k)}\big) (\smallbullet)=-\sum_{i, j=1}^{m} \theta^{j} \wedge\big(e_{i}\mathop{\lrcorner} \mathrm{R}(e_{j}, e_{i})(\smallbullet)\big)-\sum_{i,j=1}^m e_i\big(e_j(h)\big)\Big( \theta^j\wedge \left(e_i \mathop{\lrcorner} \smallbullet \right) \Big)
$$
for any orthonormal frame $\left(e_{i}\right)_{1\leq i \leq  m}$ with corresponding dual frame $(\theta^j)_{1\leq j\leq  m}$ (see Theorem \ref{pre-theorem1} below). When $k=1$, $\mathscr{R}^{(1)}-(\Hess h)^{(1)}=\Ric-\Hess h$.

Let $\lambda_k(x)$ be the lowest eigenvalue of $(\mathscr{R}^{(k)}-(\Hess  h)^{(k)})(x)$ for $x\in M$. We use the notation
$$
V_k(x):=\lambda^{-}_k(x)=(|\lambda_k(x)|-\lambda_k(x)) / 2 .
$$
Let $P_{t}^{V_k}$ be the semigroup ${\e}^{-t(\Delta_{\mu}-V_k)}$ which has a smooth integral kernel denoted by $p_{t}^{V_k}(x,y)$.

\begin{definition}\label{strongly-positive}
We say that $\Delta_{\mu}-V_k+\sigma$ for some constant $\sigma$ is {\it strongly positive} if the following  condition holds: there exists $A<1$ such that for all $f \in C_c^{\infty}(M)$,
$$
\int_{M} (V_k-\sigma) |f|^{2} \d \mu \leq A \int_{M}|\nabla f|^{2}\, \d \mu.
$$
\end{definition}
We remark that the strong positivity condition  has  its origin in
the  Hardy inequality, see the introduction of \cite{CZ07}.

The following theorem is our first main result.
\begin{theorem}\label{main-them1}
 Let $\sigma_1$, $\sigma_2$ and  $\sigma_2$ be positive constants such that the following three
conditions hold:
\begin{enumerate}[\rm i)]
\item   (local) volume doubling property: for $\alpha>1$,
\begin{align}\label{LD}
\mu(B(x,  \alpha r)) \leq C \mu(B(x, r)) \,\alpha^{m}\exp (\sigma_1(\alpha -1)r),   \quad x\in M,  \tag{{\bf LD}}
\end{align}
holds for all $r>0$ and  some constant $C>0$;

\item  local off-diagonal upper bound of the heat kernel $p^{V_k}_t$:  
\begin{align}\label{Upper-bound}
p^{V_k}_t(x, x)\leq \frac{C\e^{\sigma_2 t}}{\mu(B(x,\sqrt{t}))},\quad x\in M,  \tag{{\bf UE}}
\end{align}
for all $t>0$ and  some constant $C>0$;

\item  
 the operator $\Delta_{\mu}-V_k+\sigma_3$ is strongly positive.
\end{enumerate}
Then there exists a positive constant $\sigma$ depending on $\sigma_1$, $\sigma_2$ and $\sigma_3$  such that the covariant Riesz transform $\nabla (\Delta_{\mu}^{(k)}+\sigma)^{-1/2}$ is bounded in $L^p$ for $p\in (1,2]$. In particular,  $\sigma_1=\sigma_2=\sigma_3=0$ implies   $\sigma=0$.
\end{theorem}

The upper bound estimate of Schr\"{o}dinger heat kernel with an increasing exponential factor have been well studied in 
\cite{Sturm,Simon}. Often, inequality  \eqref{Upper-bound} appears without the increasing exponential factor (i.e. $\sigma_2=0$), which requires stronger
conditions on the curvature and the potential $V_k$. In the following, we use the result from  \cite{Zhang00} to consider the global 
covariant Riesz transform, i.e. $\sigma=0$.

 Let $P_t=\e^{-\Delta_{\mu} t}$ be the semigroup generated by $-\Delta_{\mu}$ and 
$p_t(x,y)$ the corresponding heat kernel with respect to the measure $\mu$. 
Our second main result is the following Theorem \ref{them2}.  It has been proved in \cite{Zhang00} that the doubling volume property \eqref{D} and the on-diagonal estimate \eqref{on-diagonal-Upper-bound} in this theorem, together with condition \eqref{KV}, imply that
\begin{align*}
p_{t}^{V_k}(x, x) \leq \frac{C}{\mu(B(x, \sqrt{t}))}.
\end{align*}
Thus the following result is a consequence of Theorem \ref{main-them1} for $\sigma_1=\sigma_2=\sigma_3=0$.

\begin{theorem}\label{them2}

Suppose the following conditions hold:
\begin{enumerate}[\rm i)]
\item   volume doubling property: for $\alpha>1$,
\begin{align}\label{D}
\mu(B(x,  \alpha r)) \leq C \mu(B(x, r)) \alpha^{m} \qquad \tag{{\bf D}}
\end{align}
holds for some constant $C>0$ and all $x \in M, r>0$;

\item on-diagonal upper bound of the heat kernel:  there exists  constants $C, \delta>0$ such that 
\begin{align}\label{on-diagonal-Upper-bound}
p_t(x, x)\leq \frac{C}{\mu(B(x, \sqrt{t}))}  \tag{{\bf U}}
\end{align}
for all $t>0$ and  $x \in M$, and 
\begin{align}\label{KV}
K(V_k) \equiv \sup _{x \in M} \int_{0}^{\infty} \int_{M} \frac{1}{\mu(B(x, \sqrt{s}))} \e^{-\rho^{2}(x, y) / s} V_k(y)\, \mu(\d y)\,  \d s<\delta;
\end{align}

\item  
 the operator $\Delta_{\mu}-V_k$ is strongly positive on $\Omega_c^k$.
\end{enumerate}
 Then the Riesz transform $\nabla (\Delta_{\mu}^{(k)})^{-1/2}$ on $\Omega^k_{L^p}$ is bounded in $L^p$ for $p\in (1,2]$.
\end{theorem}\goodbreak

Let us compare to known results.
In \cite{ThW-04} for the usual Riemannian manifold, i.e. $h=0$, under the assumption that $\nabla {\rm R} + \nabla \mathscr{R}^{(k)}=0$,  $\mathscr{R}^{(k)}\geq  0$  and  the doubling volume property, it is shown that $\nabla (\Delta^{(k)})^{-1/2}$ has the weak type $(1,1)$ property. In the above theorem, if
$\Delta_{\mu}-V_k$ is strongly positive on $\Omega^k_c$, the lower bound of $\mathscr{R}^{(k)}$
can be relaxed  and  no condition on the derivative of curvature is needed. From this point of view,  our results also improve 
the recent work of Baumgarth, Devyver and G\"{u}neysu  \cite{BDG-21} on the covariant Riesz transform on $k$-forms for $p\in (1,2]$.

It has been observed that the curvature-dimension condition implies the local volume doubling condition (see \cite{GW01}). 
For the localization argument towards the boundedness of Riesz transform, we need the local doubling volume property with respect to $\mu$, which is related  to the following curvature-dimension condition.  
Assume that
\begin{align}\label{eqn-CD}
\Gamma_{2}(f, f):=-\frac{1}{2} \Delta_{\mu}|\nabla f|^{2}+(\nabla \Delta_{\mu} f, \nabla f)_g \geq-K_0|\nabla f|^{2}+\frac{1}{n}(\Delta_{\mu} f)^{2}, \tag{{\bf CD}}
\end{align}
where $K_0 \in \mathbb{R}$ and $n \geq m$ provide a curvature lower bound and a dimension upper bound of $\Delta_{\mu}$, respectively. In the case $\nabla h=0$ this condition is equivalent to $\Ric\geq-K_0$ and then the curvature-dimension condition holds for $n=m$. When $\nabla h \neq 0$ however, typically $n$ is larger than $m$. Indeed, the curvature-dimension condition can be written as
$$
\operatorname{Ric}_{h}^{(n-m)}(X, X) \geq-K_0|X|^{2}, \quad X \in T M
$$
where for $\alpha>0$, the $\alpha$-Ricci curvature of the weighted Laplacian $\Delta_{\mu}$ is defined as
$$
\operatorname{Ric}_{h}^{\alpha}:=\operatorname{Ric}-\text { Hess }{h}-\frac{1}{\alpha} \nabla h \otimes \nabla h.
$$
This condition implies that
\begin{align}\label{eqn-Ric}
\operatorname{Ric}_{h}(X, X):=\operatorname{Ric}(X,X)- (\text {Hess }{h})(X,X) \geq-K_0|X|^{2}. \tag{{\bf Ric}}
\end{align}

Assuming the curvature-dimension curvature condition \eqref{eqn-CD} then in particular the local doubling assumption with respect to $\mu$ holds, see \cite{GW01, Qian} for details, i.e., there exists a constant $L>0$ such that
\begin{align}\label{eqn-GD}
\mu(B(y, \alpha r)) \leq  C \mu(B(y, r)) \alpha^{m} \exp (L(\alpha-1) r), \quad y\in M,\ r>0,\ \alpha>1.
\end{align}


 \begin{assumption}\label{Lower-Rich}
 For  $k\in \mathbb{N}$ and $k\geq 2$,
there exists $K\in \R$ such that
$$
-K=\min \left\{\left(\big(\mathscr{R}^{(k)}-(\text {Hess }h)^{(k)}\big) v, v\right)_g: v \in \Lambda^{(k)} T_{x} M,\  |v|=1,\  x \in M\right\} .
$$
\end{assumption}

Obviously, Assumption \ref{Lower-Rich}  implies that
$\mathscr{R}^{(k)}-(\text {Hess }h)^{(k)}$ is bounded below by $-K$ so that
$-V_k+K^+\geq 0$, hence the operator 
$\Delta_{\mu}-V_k+K^{+}$ is  strongly positive.
Assume that \eqref{eqn-CD} holds for some constant $K_0$. On the one hand \eqref{eqn-CD} implies the local volume doubling property. On the other hand, it implies  the lower Ricci curvature bound \eqref{eqn-Ric}, which is further used to derive the Gaussian type estimate of $p_t(x,y)$ (see \cite[Theorem 2.4.4]{Wbook14}). We then conclude that for any $\alpha \in\left(0,1/4\right)$ there exist constants $C_{1}(\alpha), C_{2}(\alpha)>0$ such that
\begin{align}\label{Gaussian}
|p^{V_k}_t(x,y)|\leq \e^{K^+t}p_t(x,y) \leq \frac{C_{1}(\alpha)}{\mu(B(x, \sqrt{t}))} \exp \left(-\frac{\alpha \rho(x, y)^{2}}{t}+(K^++C_{2}(\alpha)) t\right) 
\end{align}
for $t> 0$. As a consequence, we have the following corollary  from Theorem \ref{main-them1} directly.
\begin{corollary}\label{them3}
Assume \eqref{eqn-CD} holds for some $K_0\geq 0$. Then 
there exists a constant $\sigma>0$ such that the Riesz transform $\nabla (\Delta_{\mu}^{(1)}+\sigma)^{-1/2}$ 
is bounded in $L^p(\mu)$ for $1<p\leq 2$. If in addition  Assumption \ref{Lower-Rich} holds for some $k\geq 2$,
then there exists $\sigma>0$ such that the Riesz transform $\nabla (\Delta_{\mu}^{(k)}+\sigma)^{-1/2}$ on $\Omega_{L^p}^{k}$
is bounded in $L^p(\mu)$ for  $1<p\leq 2$.
\end{corollary}

As explained at the beginning, the result of Theorem \ref{them3} implies the Calder\'on-Zygmund inequality for $1< p< 2$.
 We say that an
  $L^p(\mu)$-Calder\'{o}n-Zygmund inequality holds on $M$ if there exist
  two constants $C_1,C_2>0$ such that
  \begin{equation}\label{CZ}
    \big\|\Hess(\varphi)\big\|_{p}\leq C_1\|\varphi\|_{p}+C_2\|\Delta_{\mu} \varphi\|_{{p}} \tag{{\bf CZ}$_{\mu}(p)$}
  \end{equation}
  for every function $\varphi\in C_c^{\infty}(M)$. We denote this inequality by \ref{CZ}. G\"uneysu and Pigola \cite{GP-15} observed that under Calder\'{o}n-Zygmund inequalities,
if $M$ is geodesically complete and admits a sequence of
Laplacian cut-off functions (this is the case e.g. if $M$ has non-negative Ricci curvature; for more general curvature conditions see \cite{Gu16} and \cite{BS18}), then $H_0^{2,p}(M)=H^{2,p}(M)$ holds for
all $1<p<\infty$. We refer the reader to~\cite{GP-2019} for further applications of Calder\'{o}n-Zygmund inequalities. 

In general, \ref{CZ} inequalities may hold or fail on $M$, depending on the
underlying Riemannian geometry, which leads to the question which
geometric assumptions on $M$ guarantee \ref{CZ} and how the \ref{CZ}-constants $C_1,C_2$ depend on the geometric entities.
  In \cite{GP-15} two methods appear for attacking Calder\'{o}n-Zygmund inequalities: the
first one depends on appropriate elliptic estimates under conditions
on harmonic bounds of the injectivity radius, while the second one
uses boundedness results for the covariant Riesz transform in $L^p$ for $1<p\leq 2$ from
\cite{ThW-04}.  Whereas conditions on harmonic bounds of the
injectivity radius are usually difficult to verify, the second
approach relies on probabilistic covariant derivative formulae for
heat semigroups and has the advantage to avoid assumptions on the
injectivity radius. Along the main idea of this second method in \cite{GP-15}, 
 Theorem \ref{them2} permits
 to establish \ref{CZ} for $1<p\leq 2$ on weighted manifolds along
 the same approach  but only using
 the curvature-dimension condition.

 \begin{theorem}\label{them-CZ}
  Let $(M,g)$ be a complete Riemannian manifold satisfying
  \eqref{eqn-CD}. Let $1<p<2$ be fixed. Then there exists a constant
  $\sigma>0$ such that the operator $\Hess (\Delta_{\mu}+\sigma)^{-1}$ is
  bounded in $L^p(\mu)$, and in particular~\ref{CZ} holds.
\end{theorem}

Comparing Theorem \ref{them-CZ} with existing results on \ref{CZ},
it should be pointed out that  the result is valid without any
injectivity radius assumptions and boundedness of
$\|R\|_{\infty}$ and $\|\nabla R\|_{\infty}$ as in \cite{GP-15}.
Our  result extends \cite{CCT}
to the weighted manifold  by only requiring the curvature-dimension condition.

The paper is organized as follows. In Section 2 we present $L^2$ and $L^1$ weighted derivative estimates for the heat kernel on differential forms (see Theorems \ref{Dkernel-Gaussian} and \ref{cor-1}). These estimates are applied in Section~\ref{Sect3} to study the $L^{p}$-boundedness ($1<p \leq  2$) of Riesz transforms for differential forms on Riemannian manifolds with a metric connection (see Theorem \ref{main-them1}). Moreover, Theorem \ref{them3} gives a local version of covariant Riesz transform on Riemannian forms. As application, Theorem \ref{them3}
is used to obtain the Calder\'{o}n-Zygmund inequalities for $p\in (1,2]$.


\smallskip{\bf Acknowledgements.} The authors are indebted to Batu
G\"uneysu, Stefano Pigola and Giona Veronelli for very helpful
comments on the topics of this paper.\goodbreak

\section{Heat kernel  estimates }

\subsection{Preliminaries}

Let us first recall the interior product. 
\begin{definition}\label{def1}
 The interior product $X\mathop{\lrcorner} a\in\Omega^{k-1}$ corresponds to the contraction of $a\in\Omega^{k}$ with a vector field $X \in \Gamma(T M)$ and is defined as 
$$
X\mathop{\lrcorner} a\left(X_{1}, \ldots, X_{k-1}\right):=a\left(X, X_{1}, \ldots, X_{k-1}\right), \quad \forall X_{1}, \ldots, X_{k-1} \in \Gamma(T M).
$$
The interior product is an anti-derivation, i.e.,
$$
X\mathop{\lrcorner}\,(a \wedge b)=(X\lrcorner\, a) \wedge b+(-1)^{k} a \wedge(X\mathop{\lrcorner}b) \quad \forall a \in \Omega^{k}, \ b \in \Omega^1 .
$$
\end{definition}

The Weitzenb\"{o}ck formula relates the weighted Hodge-de Rham Laplacian to the weighted Bochner Laplacian on $(M, g)$. 

\begin{theorem}[Weitzenb\"{o}ck formula]\label{pre-theorem1}
 For all differential $k$-forms $\eta \in \Omega^{k}$, we have
$$
\Delta_{\mu}^{(k)} \eta=\square\eta-\nabla_{\nabla h} \eta+\mathscr{R}^{(k)}(\eta)- ({\rm Hess}\, h)^{(k)}(\eta),
$$
where $\mathscr{R}^{(k)}:\Omega^k \rightarrow \Omega^k$ is given by
$$
\mathscr{R}^{(k)}(\eta)=-\sum_{i, j=1}^{m} \theta^{j} \wedge \left(e_{i} \mathop{\lrcorner}\mathrm{R}\big(e_{j}, e_{i}\big)(\eta)\right)
$$
 and $(\Hess h)^{(k)}:\Omega^k \rightarrow \Omega^k$ by
\begin{align*}
(\Hess  h)^{(k)} (\eta)=\sum_{i,j=1}^m e_i(e_j(h))\theta^j\wedge \left(e_i \mathop{\lrcorner}\eta \right),
\end{align*}
for any orthonormal frame $\left(e_{i}\right)_{1\leq i \leq  m}$ and corresponding dual frame $(\theta^{j})_{1\leq  j \leq  m}$ such that $\nabla e_i=0$, $\nabla \theta^j=0$ and  $\theta^j(e_i)=\delta_i^j$.
\end{theorem}

\begin{proof}
It is well known that
\begin{align*}
{\rm d}^{(k)}=\sum_{j=1}^m\theta^j\wedge \nabla _{e_j}\quad\text{and}\quad  \delta_{\mu}^{(k)}(\cdot)=-\sum_{j=1}^m\e^{-h}e_j \mathop{\lrcorner}\nabla _{e_j}(\e^{h}(\cdot)). 
\end{align*}
Let $\Delta ^{(k)}$ be the usual Hodge Laplacian acting on $k$-form. Since orthonormal frames $(e_{i})_{1\leq  i \leq  m}$ and  dual frames $\left(\theta^{j}\right)_{1\leq j \leq  m}$ satisfy $\nabla e_i=0$ and $\nabla \theta^j=0$, we obtain for $\eta\in \Omega^k$, using the summation convention,
\begin{align*}
\Delta_{\mu}^{(k)}\eta&=-\e^{-h}e_j \mathop{\lrcorner} \nabla _{e_j}(\e^{h}(\theta^i\wedge \nabla _{e_i}\eta))-\theta^i\wedge \nabla _{e_i}(\e^{-h}e_j \mathop{\lrcorner} \nabla _{e_j}(\e^{h}\eta))\\
&=-e_j(h)e_j \mathop{\lrcorner} (\theta^i\wedge \nabla _{e_i}\eta)-e_j \mathop{\lrcorner} \nabla _{e_j}(\theta^i\wedge \nabla _{e_i}\eta)\\
&\quad -\theta^i\wedge \nabla _{e_i}(e_j(h)e_j\mathop{\lrcorner} \eta)-\theta^j\wedge \nabla _{e_j}(e_j \mathop{\lrcorner} \nabla _{e_j} \eta)\\
&=-e_j(h)e_j \mathop{\lrcorner} (\theta^i\wedge \nabla _{e_i}\eta)-e_i(e_j(h))\theta^i\wedge (e_j\mathop{\lrcorner} \nabla _{e_j}\eta)-e_j(h)\theta^i\wedge \nabla _{e_i}(e_j\mathop{\lrcorner} \eta)+\Delta^{(k)}\eta\\
&=-e_j(h)\nabla _{e_j}\eta+e_j(h) (\theta^i\wedge (e_j \mathop{\lrcorner}\nabla _{e_i}\eta))-e_i(e_j(h))\theta^i\wedge (e_j\mathop{\lrcorner} \eta)-e_j(h)(\theta^i\wedge (e_j\mathop{\lrcorner}\nabla _{e_i} \eta))\\
&\quad +\square \eta-\mathscr{R}^{(k)}(\eta)\\
&=-e_j(h)\nabla _{e_j}\eta-e_i(e_j(h))\theta^i\wedge (e_j\mathop{\lrcorner} \eta)+\square \eta-\mathscr{R}^{(k)}(\eta)\\
&= -\nabla _{\nabla h}\eta-({\rm Hess}\, h)^{(k)}(\eta) +\square \eta-\mathscr{R}^{(k)}(\eta),                    
\end{align*}
where the last equation follows from the fact that
\begin{equation*}
\Delta^{(k)}\eta=\square \eta-\mathscr{R}^{(k)}(\eta).\qedhere
\end{equation*}
\end{proof}

By the usual abuse of notation, the  corresponding self-adjoint realizations will again
be denoted by the same symbol, i.e. $\Delta_{\mu}$ and $\Delta^{(k)}_{\mu}$
respectively. By local parabolic regularity, for all square-integrable $k$-forms $a \in \Omega^k_{L^2}$, the time-dependent $k$-form
$$
(0,\infty) \times M \ni(t, x) \mapsto \mathrm{e}^{-\Delta_{\mu}^{(k)}t} a \in \Lambda^{k} T_{x}^{*} M
$$
has a smooth representative which extends smoothly to $[0,\infty)\times M$ if $a$ is smooth.
In addition, there exists a unique smooth heat kernel of $\mathrm{e}^{-\Delta_{\mu}^{(k)}t}$ with respect to the measure $\mu$, which is understood as a map
\begin{align*}
(0,\infty)\times M \times M \ni (t,x,y)\mapsto \e^{-\Delta_{\mu}^{(k)}t}(x,y)\in {\rm Hom}(\Lambda^k T^*_yM, \Lambda^k T^*_xM)
\end{align*}
such that
\begin{align*}
\e^{-t\Delta_{\mu}^{(k)}}a(x)=\int_M \e^{-t\Delta_{\mu}^{(k)}}(x,y)a(y)\,\mu(\d y).
\end{align*}
Let
$$
-K:=\min \left\{\big((\mathscr{R}-\Hess h)^{(k)} v, v\big)_g\colon v \in \Lambda^{k} T_{x} M,\ |v|=1,\ x \in M\right\} .
$$
Recall the notation  $\lambda_k(x)$ defined as the smallest eigenvalue of $(\mathscr{R}-\Hess h)^{(k)}(x), x \in M$ and let
$$
V_k(x)=\lambda_k^{-}(x)=(|\lambda_k(x)|-\lambda_k(x)) / 2.
$$
Then by \cite{HSU80},
$$
\big|\exp \left(-t \Delta_{\mu}^{(k)}\right)(x, y)\big| \leq p^{V_k}_{t}(x, y).
$$
We conclude that to estimate $\big|\exp \left(-t \Delta_{\mu}^{(k)}\right)(x, y)\big|$,
it suffices to estimate the Schr\"{o}dinger heat kernel $p^{V_k}_{t}(x, y)$. There is 
 a lot of previous work dealing with
Schr\"{o}dinger heat kernels on manifolds, see for instance \cite{CZ07,Batu:2017,Sturm,Zhang00,Zhang01}.

\begin{theorem}
  \label{Gaussian-estimate}
  Let $M$ be a complete non-compact Riemannian manifold satisfying  \eqref{LD} and \eqref{Upper-bound}.   Then for any $\alpha \in (0,1 / 4)$, there exists $ \tilde{\sigma}>0$  depending only on   the constants in  \eqref{LD} and \eqref{Upper-bound} and a constant $C>0$ such that
$$
\big|\exp{(-t\Delta_{\mu}^{(k)})}(x, y)\big| \leq \frac{C\e^{\tilde{\sigma} t}}{\mu(B(y, \sqrt{t}))} \exp \left(-\alpha \rho(x, y)^{2} / t\right),\quad \forall x,y\in M, \, t>0.
$$
If $\sigma_1=0$ and $\sigma_2=0$, then $\tilde{\sigma}=0$.

\end{theorem}
\begin{proof}
Let $P_t^{V_k}$ be the semigroup generated by the operator $-\Delta_{\mu}+V_k$ and   $p_t^{V_k}(x,y)$ the corresponding heat kernel. 
We recall that
\begin{align*}
\left|\exp{(-t\Delta_{\mu}^{(k)})}(x, y)\right| \leq  p_t^{V_k}(x,y),\quad   x,y\in M,\  t>0.
\end{align*}
From the assumptions   \eqref{LD} and \eqref{Upper-bound}, one can derive that $p_t^{V_k}$ satisfies 
the Gaussian estimate:
\begin{align*}
p^{V_k}_t(x,y)\leq \frac{C\e^{\tilde{\sigma} t}}{\mu(B(y, \sqrt{t}))} \exp \left(-\alpha \rho(x, y)^{2} / t\right)
\end{align*}
for $x,y \in M$ and $t>0$, by the same argument as in \cite{Gr97}; see \cite[Theorem 3.1]{CZ07} for a similar argument.
\end{proof}

\begin{lemma}\label{lem3}
If the local volume doubling property \eqref{LD} holds, then for any $\gamma>0$, there exist positive constants $C_{\gamma}$ and $\tilde{c}:=\sigma_1^2/2\gamma$ such that
\begin{align}\label{eqn-rho}
\int_{\rho(x, y) \geq \sqrt{t}} \mathrm{e}^{-2 \gamma \frac{\rho^{2}(x, y)}{s}} \mu(\d x) \leq C_{\gamma}\, \mu\big(B(y, \sqrt{s})\big) \,\mathrm{e}^{-\gamma t / s} \mathrm{e}^{\tilde{c} s}
\end{align}
for $s,t>0$ and $x,y\in M$.
\end{lemma}
\begin{proof}
  By \eqref{eqn-GD}, it is easy to see that for all $\gamma>0$,
  $s,t>0$ and $y\in M$, there exist two positive constants
  $C_{\gamma}$ (depending on $\gamma$ and the constants in
  \eqref{eqn-GD}) and $\tilde{c}=\sigma_1^2/2\gamma$ such that
  \begin{align}\label{eqn-ele1}
    \int_{\rho(x,y)\geq \sqrt{t}}\e^{-2\gamma\frac{\rho^2(x,y)}{s}}\,\mu(\d x)
    &\leq \e^{-\gamma t/s}\int_M \e^{-\gamma \frac{\rho^2(x,y)}{s}}\,\mu(\d x) \notag\\
    &\leq  \e^{-\gamma t/s}\sum_{i=0}^{\infty}\mu(B(y,(i+1)\sqrt{s}))\e^{-\gamma i^2}\notag\\
    &\leq C  \e^{-\gamma t/s}\mu(B(y,\!\sqrt{s})) \sum_{i=0}^{\infty}
      (i+1)^{m+1} \e^{-\gamma i^2}\e^{\sigma_1i\sqrt{s}} \notag\\
    &\leq  C  \e^{-\gamma t/s}\mu(B(y,\!\sqrt{s})) \sum_{i=0}^{\infty}
      (i+1)^{m+1} \e^{-\gamma i^2}\e^{\gamma i^2/2+\sigma_1^2s/(2\gamma)}\notag\\
    & \leq  C  \e^{-\gamma t/s}\e^{\sigma_1^2s/(2\gamma)}\mu(B(y,\!\sqrt{s})) \sum_{i=0}^{\infty}
      (i+1)^{m+1} \e^{-\gamma i^2/2}\notag\\
    & \leq C_{\gamma} \mu(B(y,\!\sqrt{s})) \,\e^{-\gamma t/s} \e^{\tilde{c}s},
  \end{align}
  where the third inequality comes from condition \eqref{eqn-GD}. 
\end{proof}

By means of this estimate, we obtain immediately the following consequence.

\begin{theorem}\label{L2-estimate}
 Let $M$ be a complete non-compact Riemannian manifold satisfying  \eqref{LD} and \eqref{Upper-bound}. Then  for any $\alpha\in (0,1/4)$ and $\gamma \in (0, \alpha)$, there exists some constant $C>0$ such that 
$$
\int_M\left|\exp{(-t\Delta_{\mu}^{(k)})}(x, y)\right|^2 \e^{ \frac{2\gamma \rho^2(x,y)}{t}}\, \mu(\d x) \leq \frac{C\e^{2C_0t}}{\mu\big(B(y, \sqrt{t})\big)},
$$
for all $ y \in M$ and $t>0$, where  $C_0:=\tilde{\sigma}+ \frac{1}{2}\tilde{c}$ and the constants $\tilde{\sigma}, \tilde{c}$ defined in Theorems \ref{Gaussian-estimate} and \ref{lem3} respectively.  

\end{theorem}

\begin{proof}
 Letting $t$ tend to $\infty$ in inequality \eqref{eqn-rho}, we obtain
$$
\int_{M} \mathrm{e}^{-2 \gamma \frac{\rho^{2}(x, y)}{t}} \mu(\d x) \leq C_{\gamma}\, \mu\Big(B\big(y, \sqrt{t}\big)\Big) \,\mathrm{e}^{\tilde{c} t}, \quad t>0 .
$$
By Theorem \ref{Gaussian-estimate} and Lemma \ref{lem3}, we conclude that there exists $\tilde{\sigma}>0$ depending only on the constants 
$\sigma_1$ and $\sigma_2$ and a constant $C>0$ such that 
\begin{align*}
&\int_M\left|\exp{(-t\Delta_{\mu}^{(k)})}(x, y)\right|^2 \e^{ \frac{2\gamma \rho^2(x,y)}{t}}\, \mu(\d x) \\
&\leq  C\frac{\e^{2\tilde{\sigma} t}}{\mu\big(B(y,\sqrt{t})\big)^2}\int_M \e^{ \frac{-(2\alpha- 2\gamma) \rho^2(x,y)}{t}}\, \mu(\d x) \\
&\leq  C\frac{\e^{(2\tilde{\sigma}+\tilde{c}) t}}{\mu(B\big(y,\sqrt{t})\big)}.
\end{align*}
We then complete the proof.
\end{proof}

\subsection{$L^2$-weighted derivative estimates of  heat kernel}
In this subsection, we start the discussion under the assumption that \eqref{LD} and \eqref{Upper-bound} hold and that the operator $\Delta_{\mu}-V_k+\sigma_3$ on $\Omega^k$ is strongly 
positive. Then we have the following result about the $L^2$-weighted derivative estimate of the heat kernel.

\begin{theorem}\label{Dkernel-Gaussian}
 Let $M$ be a complete non-compact Riemannian manifold satisfying the assumptions as in Theorem \ref{main-them1}. Fix $\alpha \in\left(0, {1}/{4}\right)$ as in Theorem \ref{Gaussian-estimate}.  Then for any $0<\gamma<\alpha$,  there exists a constant $C>0$  such that 
 $$
\int_{M} \left|\nabla \exp \left(-t \Delta_{\mu}^{(k)}\right)(x, y)\right|^{2} \mathrm{e}^{2 \gamma \frac{\rho^{2}(x, y)}{t}} \mu(\d x) \leq \frac{C(1+\sigma_3 t)\e^{2C_0 t}}{t\mu(B(y, \sqrt{t}))}
$$
for  all $y \in M$, $t>0$, where the constant $C_0$ is  defined as in Theorem \ref{L2-estimate} 
\end{theorem}

\begin{proof} 
For $R>0$, let $\psi$ be a $C^2$ function on $\R^+$ such that $\psi(r)=1$ for $r\in [0,R]$, $\psi(r)=0$ for $r>2R$ and $\|\psi\|_{\infty}\leq 1,\ \ \|\psi'\|_{\infty}\leq \frac{c\psi^{1/2}}{R}$ for some positive constant $c>0$ (see \cite{LY86}). An argument of Calabi, which is also used in
\cite{Cheng-Yau}, allows us  to assume without loss of generality that $\mathrm{e}^{\frac{2 \gamma \rho^2(\cdot,y)}{t}}\psi(\rho(\cdot,y))$ for $y\in M$ is smooth. 
According to the integration by parts formula, we have 
\begin{align*}
&\int_{M}\left|\nabla \exp \left(-t \Delta_{\mu}^{(k)}\right)(x, y)\right|^{2} \mathrm{e}^{\frac{2 \gamma \rho^2(x,y)}{t}} \psi(\rho(x,y)) \mu(\d x) \\
&=\int_{M}\left(\psi(\rho(x,y)) 4 \gamma \frac{\rho(x, y)}{t} + \psi'(\rho(x, y))\right)\left(\nabla_{\nabla \rho} \exp \left(-t \Delta_{\mu}^{(k)}\right)(x, y),  \exp \left(-t \Delta_{\mu}^{(k)}\right)(x, y)\right)_g \mathrm{e}^{ \frac{2\gamma \rho^{2}(x, y)}{t}} \mu(\d x) \\
&\quad +\int_{M}\psi(\rho(x,y)) \left(\nabla^{*}_{\mu} \nabla \exp \left(-t \Delta_{\mu}^{(k)}\right)(x, y), \exp \left(-t \Delta_{\mu}^{(k)}\right)(x, y)\right)_g \mathrm{e}^{\frac{2 \gamma \rho^{2}(x, y)}{t}} \mu(\d x) =:{\rm I}+{\rm II}.
\end{align*}
Then there exists $\alpha>\gamma^{\prime}>\gamma>0$ such that
\begin{align*}
&{\rm I}=\int_{M} 4 \gamma \frac{\rho(x, y)}{t}\left(\nabla_{\nabla \rho} \exp \left(-t \Delta_{\mu}^{(k)}\right)(x, y), \exp \left(-t \Delta_{\mu}^{(k)}\right)(x, y)\right)_g \mathrm{e}^{ \frac{2\gamma \rho^{2}(x, y)}{t}}\psi(\rho(x,y))\, \mu(\d x)\\
&\quad +\int_{M} \psi'(\rho(x, y))\left(\nabla_{\nabla \rho} \exp \left(-t \Delta_{\mu}^{(k)}\right)(x, y),  \exp \left(-t \Delta_{\mu}^{(k)}\right)(x, y)\right)_g \mathrm{e}^{ \frac{2\gamma \rho^{2}(x, y)}{t}} \mu(\d x) \\
&=\int_{M} 4 \gamma \frac{\rho(x, y)}{t}\left(\nabla \exp \left(-t \Delta_{\mu}^{(k)}\right)(x, y), \left({\rm d}  \rho \otimes \exp \left(-t \Delta_{\mu}^{(k)}\right)\right)(x, y)\right)_g \mathrm{e}^{\frac{2 \gamma \rho^{2}(x, y)}{t}}\psi(\rho(x,y))\, \mu(\d x) \\
&\quad +\frac{c}{R} \int_{M}\left|\left(\nabla \exp \left(-t \Delta_{\mu}^{(k)}\right)(x, y), \left({\rm d}  \rho \otimes \exp \left(-t \Delta_{\mu}^{(k)}\right)\right)(x, y)\right)_g \right| \mathrm{e}^{\frac{2 \gamma \rho^{2}(x, y)}{t}}\psi(\rho(x,y))^{1/2}\, \mu(\d x) \\
&\leq \frac{C}{\sqrt{t}} \int_{M}\left|\nabla \exp \left(-t \Delta_{\mu}^{(k)}\right)(x, y)\right| \cdot\left|\exp \left(-t \Delta_{\mu}^{(k)}\right)(x, y)\right| \mathrm{e}^{\frac{2 \gamma^{\prime} \rho^{2}(x, y)}{t}}\psi(\rho(x,y))\, \mu(\d  x) \\
&\quad +\frac{c}{R} \int_{M}\left|\nabla \exp \left(-t \Delta_{\mu}^{(k)}\right)(x, y)\right| \cdot\left|\exp \left(-t \Delta_{\mu}^{(k)}\right)(x, y)\right| \mathrm{e}^{\frac{2 \gamma \rho^{2}(x, y)}{t}}\psi(\rho(x,y))^{1/2}\, \mu(\d x) \\
&\leq \frac{C}{\sqrt{t}}\left(\int_{M}\left|\nabla \exp \left(-t \Delta_{\mu}^{(k)}\right)(x, y)\right|^{2} \mathrm{e}^{\frac{2 \gamma \rho^{2}(x, y)}{t}}\psi(\rho(x,y))\, \mu(\d  x)\right)^{1 / 2}\\
&\qquad \quad \times \left(\int_{M}\left|\exp \left(-t \Delta_{\mu}^{(k)}\right)(x, y)\right|^{2} \mathrm{e}^{\frac{\left(4 \gamma^{\prime}-2 \gamma\right) \rho^{2}(x, y)}{t}}\psi(\rho(x,y))\, \mu(\d  x)\right)^{1 / 2} \\
&\quad +\frac{c}{R}\left(\int_{M}\left|\nabla \exp \left(-t \Delta_{\mu}^{(k)}\right)(x, y)\right|^{2} \mathrm{e}^{\frac{2 \gamma \rho^{2}(x, y)}{t}}\psi(\rho(x,y))\, \mu(\d  x)\right)^{1 / 2}\\
&\qquad \quad \times \left(\int_{M}\left|\exp \left(-t \Delta_{\mu}^{(k)}\right)(x, y)\right|^{2} \mathrm{e}^{\frac{2 \gamma \rho^{2}(x, y)}{t}}\, \mu(\d  x)\right)^{1 / 2} \\
&\leq \frac{1-A}{2}\int_{M}\left|\nabla \exp \left(-t \Delta_{\mu}^{(k)}\right)(x, y)\right|^{2} \mathrm{e}^{\frac{2 \gamma \rho^{2}(x, y)}{t}}\psi(\rho(x,y))\, \mu(\d x)\\
&\quad +\frac{C^{2}}{(1-A) t}\int_{M}\left|\exp \left(-t \Delta_{\mu}^{(k)}\right)(x, y)\right|^{2} \mathrm{e}^{\left(4 \gamma^{\prime}-2 \gamma\right)\frac{ \rho^{2}(x, y)}{t}}\psi(\rho(x,y))\, \mu(\d x)\\
& \quad +\frac{c^{2}}{(1-A) R^2}\int_{M}\left|\exp \left(-t \Delta_{\mu}^{(k)}\right)(x, y)\right|^{2} \mathrm{e}^{2 \gamma \frac{ \rho^{2}(x, y)}{t}}\, \mu(\d x)
\end{align*}
where $({\rm d} \rho)(x,y):=( {\rm d} \rho (\cdot, y))(x)$ and  $A<1$ is the constant from the strong positivity property of $\Delta_{\mu}-V_k+\sigma_3$.
Since $2\gamma'-\gamma<\alpha$, we can use the estimate in Theorem \ref{L2-estimate} to get
\begin{align*}
&\int_{M} 4 \gamma \frac{\rho(x, y)}{t}\left (\nabla \exp \left(-t \Delta_{\mu}^{(k)}\right)(x, y), \, \left(  \d  \rho \otimes \exp \left(-t \Delta_{\mu}^{(k)}\right)\right)(x, y)\right)_g\mathrm{e}^{\frac{2 \gamma \rho^{2}(x, y)}{t}}\psi(\rho(x,y))\, \mu(\d x) \\
&\leq \frac{1-A}{2}\int_{M}\left|\nabla \exp \left(-t \Delta_{\mu}^{(k)}\right)(x, y)\right|^{2} \mathrm{e}^{\frac{2 \gamma \rho^{2}(x, y)}{t}}\psi(\rho(x,y))\, \mu(\d  x)+\lf(\frac{1}{t}+\frac{1}{R^2}\r)
\frac{C\e^{2C_0 t}}{\mu(B(y,\sqrt{t}))} 
\end{align*}
for some generic constant $C$. As
$$
\Delta_{\mu}^{(k)}=\Box-\nabla_{\nabla h}+\mathscr{R}^{(k)}-(\operatorname{Hess}\, h)^{(k)},
$$
and $(\mathscr{R}^{(k)}-(\operatorname{Hess}\, h)^{(k)})(x) \geq -V_k(x)$, we then have 
\begin{align*}
{\rm II}=& \int_{M}\left(\nabla^{*}_{\mu} \nabla \exp \left(-t \Delta_{\mu}^{(k)}\right)(x, y), \exp \left(-t \Delta_{\mu}^{(k)}\right)(x, y)\right)_g \mathrm{e}^{\frac{2 \gamma \rho^{2}(x, y)}{t}}\psi(\rho(x,y))\, \mu(\d x) \\
=&\int_{M}\left(\Delta_{\mu}^{(k)} \exp \left(-t \Delta_{\mu}^{(k)}\right)(x, y), \exp \left(-t \Delta_{\mu}^{(k)}\right)(x, y)\right)_g \mathrm{e}^{\frac{2 \gamma \rho^{2}(x, y)}{t}}\psi(\rho(x,y))\, \mu(\d  x) \\
&-\int_{M}\left(\left(\mathscr{R}^{(k)}-(\mathrm{Hess}\, h)^{(k)}\right) \exp \left(-t \Delta_{\mu}^{(k)}\right)(x, y), \exp \left(-t \Delta_{\mu}^{(k)}\right)(x, y)\right)_g \mathrm{e}^{2 \gamma \frac{\rho^2(x,y)}{t}}\psi(\rho(x,y))\, \mu(\d  x) \\
\leq &\int_{M}\left(\Delta_{\mu}^{(k)} \exp \left(-t \Delta_{\mu}^{(k)}\right)(x, y), \exp \left(-t \Delta_{\mu}^{(k)}\right)(x, y)\right)_g \mathrm{e}^{\frac{2 \gamma \rho^{2}(x, y)}{t}}\psi(\rho(x,y))\, \mu(\d  x) \\
&+ \int_{M}(V_k(x)-\sigma_3)\left|\exp \left(-t \Delta_{\mu}^{(k)}\right)(x, y)\right|^{2} \mathrm{e}^{\frac{2 \gamma \rho^{2}(x, y)}{t}}\psi(\rho(x,y))\, \mu(\d  x)\\
&+\sigma_3 \int_{M}\left|\exp \left(-t \Delta_{\mu}^{(k)}\right)(x, y)\right|^{2} \mathrm{e}^{\frac{2 \gamma \rho^{2}(x, y)}{t}}\psi(\rho(x,y))\, \mu(\d x)\\
\leq &\int_{M}\left(\Delta_{\mu}^{(k)} \exp \left(-t \Delta_{\mu}^{(k)}\right)(x, y), \exp \left(-t \Delta_{\mu}^{(k)}\right)(x, y)\right)_g \mathrm{e}^{\frac{2 \gamma \rho^{2}(x, y)}{t}}\psi(\rho(x,y))\, \mu(\d x) \\
&+A \int_{M}\left|{\rm d} | \exp \left(-t \Delta_{\mu}^{(k)}\right)(x, y)| \right|^{2} \mathrm{e}^{\frac{2 \gamma \rho^{2}(x, y)}{t}}\psi(\rho(x,y))\, \mu(\d  x) \\
&
+\sigma_3 \int_{M}\left|\exp \left(-t \Delta_{\mu}^{(k)}\right)(x, y)\right|^{2} \mathrm{e}^{\frac{2 \gamma \rho^{2}(x, y)}{t}}\psi(\rho(x,y))\, \mu(\d  x).
\end{align*}
Using Kato's inequality we further obtain
\begin{align*}
{\rm II}\leq  &
\int_{M}\left(\Delta_{\mu}^{(k)} \exp \left(-t \Delta_{\mu}^{(k)}\right)(x, y), \exp \left(-t \Delta_{\mu}^{(k)}\right)(x, y)\right)_g \mathrm{e}^{\frac{2 \gamma \rho^{2}(x, y)}{t}}\psi(\rho(x,y))\, \mu(\d  x) \\
&+A \int_{M}\left|\nabla \exp \left(-t \Delta_{\mu}^{(k)}\right)(x,y)\right|^{2} \mathrm{e}^{\frac{2 \gamma \rho^{2}(x, y)}{t}}\psi(\rho(x,y))\, \mu(\d x)\\
&
+\sigma_3 \int_{M}\left|\exp \left(-t \Delta_{\mu}^{(k)}\right)(x,y)\right|^{2} \mathrm{e}^{\frac{2 \gamma \rho^{2}(x, y)}{t}} \mu(\d x).
\end{align*}
By 
Cauchy's integral formula, we get for $w\in \mathscr{B}(M)$ and $a_1, a_2 \in \Omega _{L^2}^k $, 
\begin{align*}
\Big\langle {\Delta}^{(k)}_{\mu} \mathrm{e}^{-t {\Delta}_{\mu}^{(k)}} a_{1},\,  w a_{2}\Big\rangle&=\Bigg|\int_{z:\,|z-t|=t / 2} \frac{\Big\langle\mathrm{e}^{-z {\Delta_{\mu}^{(k)}}} a_{1}, w a_{2}\Big\rangle}{(z-t)^{2}}\,\d z\Bigg| \leq(2 \pi)^{-1} \pi t \sup _{z\,:|z-t|=t / 2}\Bigg|\frac{\Big\langle\mathrm{e}^{-z {\Delta_{\mu}^{(k)}}}a_{1}, w a_{2}\Big\rangle}{(z-t)^{2}}\Bigg| \\
& \leq \frac{t}{2} \sup _{z:\,|z-t|=t/2}\left\|\,|\mathrm{e}^{-z {\Delta_{\mu}^{(k)}}}a_{1}|\, \sqrt{w}\,\right\|_{2}\left\|\,|a_{2}| \, \sqrt{w}\,\right\|_{2}(t / 2)^{-2}\\
&\leq \frac{2}{t} \left\|\,|a_{1}|\, \sqrt{w}\,\right\|_{2}\left\|\,|a_{2}| \, \sqrt{w}\,\right\|_{2},
\end{align*}
which implies  for $w(\newdot)=\mathrm{e}^{\frac{2 \gamma \rho^{2}(\newdot, y)}{t}}\psi(\rho(\cdot,y))$,
$$
\Big|\big\langle\Delta_{\mu}^{(k)} \mathrm{e}^{-t \Delta_{\mu}^{(k)} / 2} a_{1}, a_{2} \mathrm{e}^{\frac{2 \gamma \rho^{2}(\newdot, y)}{t}}\big\rangle\Big| \leq \frac{C}{t}\Big\|a_{1} \mathrm{e}^{\frac{\gamma \rho^{2}(\newdot, y)}{t}}\Big\|_{2}\Big\|a_{2} \mathrm{e}^{\frac{\gamma \rho^{2}(\newdot, y)}{t}}\Big\|_{2}.
$$
Letting $a_{1}(x)=\exp \left(-t\Delta_{\mu}^{(k)}/ 2\right)(x,y)$ and $a_{2}(x)=\mathrm{e}^{-t \Delta_{\mu}^{(k)}}(x, y)$, we then obtain by  Theorem \ref{L2-estimate}, 
\begin{align*}
&\left| -\int_{M}\left(\Delta_{\mu}^{(k)} \exp \left(-t \Delta_{\mu}^{(k)}\right)(x, y), \exp \left(-t \Delta_{\mu}^{(k)}\right)(x, y)\right)_g \mathrm{e}^{\frac{2 \gamma \rho^{2}(x, y)}{t}}\psi(\rho(x,y))\, \mu(\d x) \right| \\
&\quad\leq  \frac{C}{t}\left(\int_{M}\left|\exp \left(-t \Delta_{\mu}^{(k)}/ 2 \right)(x, y)\right|^{2} \mathrm{e}^{\frac{2 \gamma \rho^{2}(x, y)}{t}} \mu(\d x)\right)^{1 / 2} \\
&\qquad \times\left(\int_{M}\left|\exp \left(-t \Delta_{\mu}^{(k)}\right)(x, y)\right|^{2} \mathrm{e}^{\frac{2 \gamma \rho^{2}(x, y)}{t}} \mu(\d x)\right)^{1/2} \\
&\quad\leq \frac{C}{t}\left(\frac{C_1\e^{C_0 t}}{\mu(B(y, \sqrt{t/2}))}\right)^{1 / 2}\left(\frac{C_2 \e^{2C_0 t}}{\mu(B(y, \sqrt{t}))}\right)^{1 / 2} \\
&\quad\leq  \frac{C\e^{3/2 C_0 t}}{t\mu(B(y, \sqrt{t}))}.
\end{align*}
We conclude that
\begin{align*}
&\int_{B(y,R)}\left|\nabla \exp \left(-t \Delta_{\mu}^{(k)}\right)(x, y)\right|^{2} \mathrm{e}^{\frac{2 \gamma \rho^{2}(x, y)}{t}}\, \mu(\d x)\\
& \leq \int_{M}\left|\nabla \exp \left(-t \Delta_{\mu}^{(k)}\right)(x, y)\right|^{2} \mathrm{e}^{\frac{2 \gamma \rho^{2}(x, y)}{t}}\psi(\rho(x,y))\, \mu(\d x) \\
& \leq \frac{C(1+\sigma_3 t)\e^{2C_0 t}}{t \mu(B(y, \sqrt{t}))}+ \frac{C\e^{2C_0 t}}{R^2\mu(B(y,\sqrt{t}))}.
\end{align*}
 We then complete the proof by letting $R$ tend to $\infty$.\qedhere
\end{proof}

Combining Theorem \ref{Dkernel-Gaussian} with Lemma \ref{lem3}, we obtain
\begin{theorem}\label{cor-1}
 Let $M$ be a complete non-compact Riemannian manifold satisfying the same assumptions as in Theorem \ref{main-them1}.  Fix $\alpha \in\left(0, 1/4\right)$ as in Theorem \ref{Gaussian-estimate} and let $0<\gamma<\alpha$. There exists a constant $C>0$  such that
$$
\int_{\rho(x, y) \geq t^{1 / 2}}\left|\nabla \exp \left (-s \Delta_{\mu}^{(k)}\right)(x, y)\right| \mu(\d x) \leq C(1+\sqrt{\sigma_3 s}) \,\mathrm{e}^{2C_0 s-\gamma t / 2s} s^{-1 / 2}
$$
for all $y \in M$ and $s, t>0$, where $C_0$ is the same as in Theorem \ref{L2-estimate}.
\end{theorem}
\begin{proof} Let $0<\gamma<\alpha$. By Cauchy's inequality we obtain
$$
\begin{aligned}
&\int_{\rho(x, y) \geq t^{1 / 2}}\left|\nabla \exp \left\{-s \Delta_{\mu}^{(k)}\right\}(x, y)\right| \mu(\d x) \\
&\leq\left(\int_{M}\left|\nabla \exp \left\{-s \Delta_{\mu}^{(k)}\right\}(x, y)\right|^{2} \mathrm{e}^{2 \gamma \rho^{2}(x, y) / s} \mu(\d x)\right)^{1 / 2}\left(\int_{\rho(x, y) \geq t^{1 / 2}} \mathrm{e}^{-2 \gamma \rho^{2}(x, y) / s} \mu(\d x)\right)^{1 / 2} \\
&\leq  C \e^{C_0 s} \sqrt{\frac{ 1+\sigma_3 s }{s \mu(B(y, \sqrt{s}))}} \sqrt{\mu(B(y, \sqrt{s}))} \,\mathrm{e}^{-{\gamma t}/{2s}}\,\mathrm{e}^{\tilde c s/2} \\
&=C \e^{2C_0 s}\frac{\sqrt{1+\sigma_3 s } }{\sqrt{s}}\,\mathrm{e}^{-\gamma t / 2s},
\end{aligned}
$$
where the second inequality follows from Theorem \ref{Dkernel-Gaussian} and Lemma \ref{lem3}. This finishes the proof.
\end{proof}

\section{Proof of Theorem \ref{main-them1}}\label{Sect3}
Let us now present the main steps of the proof of Theorem \ref{main-them1}
and Theorem \ref{them2},
following closely the approach of \cite[Theorems 1.1 \ and  1.2]{TX-99}.
Some of the arguments have been used already in \cite{CCT} and can be taken from there.
For the convenience of the reader and for the sake of completeness
we give details here.\smallskip

The object of our interest is for  suitable $\sigma \geq 0$ the following operator on $\Omega^k$:
\begin{align}\label{eqn-defT}
T_{\sigma}^{(k)}:=\nabla\left(\Delta_{\mu}^{(k)}+\sigma\right)^{-1 / 2}=\frac{1}{\sqrt{\pi}}  \int_{0}^{\infty}\nabla \exp \left(-\Delta_{\mu}^{(k)} s\right) \frac{\e^{-\sigma s}}{\sqrt{s}}\, \d s.
\end{align}
We ignore the normalization constant $1/{\sqrt{\pi}}$ in the sequel which is irrelevant for our purpose. 
We start with the boundedness of $T_\sigma^{(k)}$ in $L^2$.

\begin{lemma}\label{lem-L2}
For $k\in \N^+$, 
suppose  $\Delta_{\mu}-V_k+\sigma$ is strongly positive on $C_c^{\infty}(M)$ for some constant $\sigma>0$. Then 
the operator $\nabla\left({\Delta}^{(k)}_{\mu}+\sigma\right)^{-1 / 2}$ on $\Omega_c^k$ is bounded in $L^2$-sense.
\end{lemma}

\begin{proof}
 We use the Bochner formula for $a\in \Omega_c^k$. According to the Weitzenb\"{o}ck formula
 \eqref{Eq:Weitzenboeck} and the
strong positivity of $\Delta_{\mu}-V_k+\sigma$, we have
\begin{align*}
\|\nabla a\|^2_2&=\langle \square_{\mu} a, a \rangle\\
&=\langle \Delta^{(k)}_{\mu} a, a \rangle -\langle (\mathscr{R}^{(k)}-(\Hess h)^{(k)})a, a \rangle\\
&\leq \langle \Delta^{(k)}_{\mu} a, a \rangle +\int_M (V_k-\sigma)|a|^2(x)\, \mu(\d x)+\sigma \|a\|^2\\
&\leq \langle \Delta^{(k)}_{\mu} a, a \rangle +A\int_M \big|{\rm d} |a|\big|^2(x)\, \mu(\d x)+\sigma \|a\|^2\\
&\leq \|(\Delta^{(k)}_{\mu})^{1/2}a\|^2_2 +A \|\nabla a\|^2_2+\sigma \|a\|^2_2,
\end{align*} 
where the last inequality comes from the Kato inequality (see e.g. \cite{HSU80}). 
This implies 
\begin{align*}
\|\nabla a\|^2_2&\leq \frac{1}{1-A} \|(\Delta^{(k)}_{\mu}+\sigma)^{1/2}a\|^2_2.
\end{align*}
We complete the proof by letting $a=(\Delta^{(k)}_{\mu}+\sigma)^{-1/2}b$ for $b\in \Omega_c^k$.
\end{proof}
 As the local version of  the Riesz transform is bounded in $L^2$, by the interpolation theorem,  the weak $(1,1)$ property for $T_{\sigma}^{(k)}$ already implies $L^p$-boundedness for all $p\in (1,2]$.
Hence we aim to study the weak $(1,1)$ property of $T_{\sigma}^{(k)}$ for $\sigma \geq 0$ suitable: there exists $c>0$ such that
\begin{align}\label{eqn-bdT}
\sup_{\lambda>0} \lambda \mu\left(\left|T_{\sigma}^{(k)} a\right|>\lambda\right) \leq c\,\mu(|a|), \quad a \in \Omega_c^k.
\end{align}
To this end, we use a version of the localization technique of
\cite[Section 4]{PTTS-2004} on the finite overlap property of $M$,
which has also been used in \cite{CCT}.\goodbreak

\begin{lemma}\textup{\cite{PTTS-2004}}\label{lem-fop}
  Assume that condition \eqref{LD} holds. There exists a countable subset
  $\mathcal{C}=\{x_j\}_{j\in \Lambda}\subset M$ such that
  \begin{itemize}
  \item [{\rm (i)}] $M=\cup_{j\in \Lambda} B(x_j,1)$;

  \item [{\rm (ii)}] $\{B(x_j,1/2)\}_{j\in\Lambda}$ are disjoint;

  \item [{\rm (iii)}] there exists $N_0\in\N$ such that for any
    $x\in M$, at most $N_0$ balls $B(x_j,4)$ contain $x$;

  \item [{\rm (iv)}] for any $c_0\geq 1$, there exists $C>0$ such that
    for any $j\in\Lambda$, $x\in B(x_j,c_0)$ and $r\in (0,\fz)$,
    \begin{align*}
      \mu\lf(B(x,2r)\cap B(x_j,c_0)\r)\le C \mu\lf(B(x,r)\cap B(x_j,c_0)\r) 
    \end{align*}
    and
    \begin{align*}
      \mu(B(x,r))\le C \mu\lf(B(x,r)\cap B(x_j,c_0)\r) 
    \end{align*}
    for any $x\in B(x_j,c_0)$ and $r\in (0,2c_0]$.
\end{itemize}
\end{lemma}

The following lemma provides the localization argument in order to prove
\eqref{eqn-bdT}.

\begin{lemma}\label{lem-bcT}
  Keeping the assumptions as in Theorem \ref{main-them1}, let
  $\mathcal{C}=\{x_j\}_{j\in\Lambda}$ be a countable subset of $M$
  having the finite overlap property as in Lemma
  \textup{\ref{lem-fop}}. Let $\sigma>2C_0$ where $C_0$ is as in
 Theorem  \ref{L2-estimate}. Suppose that there exists a constant
  $c>0$ such that
  \begin{align}\label{eqn-localT}
    \mu\lf(\{x\colon  \1_{B(x_j,2)}\,|T_{\sigma}^{(k)}a (x)|>\lambda\}\r)\leq \frac{c}{\lambda}\|\,|a|\,\|_1 
  \end{align}
  for any $j\in \Lambda$, $\lz\in (0,\fz)$ and
  $a\in \Omega_c^k$ supported in $B(x_j,1)$. Then property \eqref{eqn-bdT} holds for any
  $a\in \Omega_c^k$.
\end{lemma}

\begin{proof}
  For $j\in \Lambda$, set $B_j:=B(x_j,1)$ and let
  $\{\varphi_j\}_{j\in\Lambda}$ be a $C^{\infty}$-partition of the unity
  such that $0\le \varphi_j\le 1$ and each $\varphi_j$ is supported in $B_j$.  Then,
  for $a\in \Omega_c^k$ and $x\in M$, write
  \begin{align*}
    T_{\sigma}^{(k)}a(x)= \sum_{j\in\Lambda} \1_{2B_j}T_{\sigma}^{(k)}(a\varphi_j)(x)+\sum_{j\in\Lambda}(1-\1_{2B_j})T_{\sigma}^{(k)}(a\varphi_{j})(x),
  \end{align*}
  which yields that for any $\lambda>0$,
  \begin{align*}
   & \mu(\{x\colon |T_{\sz}^{(k)}a(x)|>\lambda\}) \\
    &\leq  \mu\left(\bigg\{x\colon \dsum_{j\in\Lambda} \1_{2B_j}|T_{\sigma}^{(k)}(a\varphi_j)(x)|>\frac{\lambda}{2}\bigg\}\right)
      + \mu\left(\bigg\{x\colon \dsum_{j\in\Lambda}(1-\1_{2B_j})|T_{\sigma}^{(k)}(a\varphi_j)(x)|>\frac{\lambda}{2}\bigg\}\right)\\
    &=:\mathrm{I}_1+\mathrm{I}_2.
  \end{align*}
  For $\mathrm{I}_1$, by Lemma \ref{lem-fop}\,(iii) and condition 
  \eqref{eqn-localT}, we have
  \begin{align}\label{lem-eI1}
    \mathrm{I}_1\le  \dsum_{j\in\Lambda}\mu\bigg(\bigg\{x:  \1_{2B_j}|T_{\sigma}^{(k)}(a\varphi_j)(x)|>\frac{\lambda}{2N_0}\bigg\}\bigg)
    \ls\frac{1}{\lambda}\|\,|a|\,\|_1
  \end{align}
  as desired, where the notation $a\ls b$ means $a\leq Cb$ for some
  constant $C$.\smallskip

  To bound $\mathrm{I}_2$, again by Lemma \ref{lem-fop}\,(iii), since
  $\varphi_j$ is supported in $B_j$, it is easy to see that
  \begin{align*}
    \dsum_{j\in\Lambda}|(1-\1_{2B_j})(x)\varphi_j(y)|\leq N_0 \1_{\{\rho(x,y)\geq 1\}}.
  \end{align*}
  Hence, according to the definition of $T_{\sigma}^{(k)}$ in \eqref{eqn-defT} and
  Theorem \ref{cor-1}, we get
  \begin{align*}
    \mathrm{I}_2
    &\leq \frac{2}{\lz} \dsum_{j\in\Lambda} \lf\|\,\big|\lf(1-\1_{2B_j}\r)
     T_{\sigma}^{(k)}\lf(a\varphi_j\r)\big|\,\r\|_{1}
    \\
    &\ls \frac{1}{\lz} \int_M\lf(\int_0^{\infty}\frac{\e^{-\sigma t}}{\sqrt{t}} \int_M \lf|\nabla_x\exp{(-\Delta_{\mu}^{(k)} t)}(x,y)\r|\sum_{j\in\Lambda}\big|(1-\1_{2B_j})(x)\varphi_j(y)\big|\,|a(y)|\,\mu(\d y)\,\d t\r)\mu(\d x)\\
    &\ls\frac{1}{\lz}\int_{M}\int_0^{\infty}\frac{\e^{-\sigma t}}{\sqrt{t}} \lf(\int_{\rho(x,y)\geq 1} |\nabla_x\exp{(-\Delta_{\mu}^{(k)} t)}(x,y)|\,\mu(\d x)\r)\,\d t\, |a(y)|\,\mu(\d y)\\
    &\leq \frac{1}{\lz} \int_M |a(y)|\,\mu(\d y) \int_0^{\infty}\e^{-\sigma t}(1+\sqrt{\sigma_3t})\,\e^{2C_0 t}\,\e^{-\gamma/2t}t^{-1}\,\d t,
  \end{align*}
  where $\gamma\in (0,\az)$ and $C_0$ is as in Theorem \ref{cor-1}. Thus, since $\sigma >2C_0$, we obtain
  \begin{align*}
    \mathrm{I}_2\ls \frac{1}{\lz}\int_0^{\infty}\e^{(2C_0-\sigma )t
    -{\gamma}/{2t}}\frac{(1+\sqrt{\sigma_3 t})}{t}\,\d t \,\|a\|_1\ls \frac{1}{\lz}\|a\|_1
  \end{align*}
where as usual $\|a\|_1\:=\|\,|a|\,\|_1$.  This combined with the estimate of $I_1$ in \eqref{lem-eI1}
  finishes the proof of Lemma \ref{lem-bcT}.
\end{proof}

We now turn to the proof of property \eqref{eqn-localT}, where we remove the subscript $j$ and
write $B$ for each $B(x_j,1)$ for simplicity.  Let $c_0\geq 1$. By Lemma
\ref{lem-fop}(iv), we have that $(c_0B,\mu,\rho)$ is a metric measure
subspace satisfying the {\it volume doubling property} that there
exists $C_D\ge 1$ such that
\begin{align}\label{eqn-D}
  \mu\lf(B(x,2r)\cap c_0B\r)\leq C_D \,\mu\lf(B(x,r)\cap c_0B\r) \tag{\bf D}
\end{align}
for all $x\in c_0B$ and $r>0$.\smallskip

We also use the following Calder\'on-Zygmund decomposition from
\cite{CW-book}, where $\mathcal{X}$ will replace $c_0B$.

\begin{lemma}[\cite{CW-book}]\label{lem-czdecom}
  Let $(\mathcal{X},\nu,\rho)$ be a metric measure space satisfying
  \eqref{LD}.  Let
  $f\in L^1(\mathcal{X})$ and $\lz\in (0,\fz)$. Assume
  $\|f\|_{L^1}< \lz \nu(\mathcal{X})$.  Then $f$ has a decomposition
  of the form $$\textstyle f=g+h=g+\sum_ih_i$$ such that
  \begin{enumerate}[{\rm(a)}]
  \item $g(x)\leq C\lambda$ for almost all $x\in M$;
  \item there exists a sequence of balls $\tilde{B}_i=B(x_i,r_i)$ so that the
    support of each $h_i$ is contained in $\tilde{B}_i$:
    \begin{align*}
      \int_{\mathcal{X}} |h_i(x)|\, \nu(\d x)\leq C\lambda \nu(\tilde{B}_i)\quad \text{and}\quad
      \int_{\mathcal{X}} h_i(x)\, \nu(\d x)=0;
    \end{align*}
  \item
    $\displaystyle\sum_i \nu(\tilde{B}_i)\leq
    \frac{C}{\lambda}\int_{\mathcal{X}} |f(x)|\, \nu(\d x)$;
  \item there exists $k_0\in \mathbb{N}^*$ such that each point of $M$
    is contained in at most $k_0$ balls $\tilde{B}_i$.
  \end{enumerate}
\end{lemma}

\begin{lemma}\label{lem-le}
   Let $M$ be a complete non-compact Riemannian manifold satisfying  \eqref{LD} and \eqref{Upper-bound}.   Let $\lz\in (0,\fz)$ and
  $f=|a|\in L^1(B)$ be as in Lemma \textup{\ref{lem-czdecom}}. Let furthermore
  $\{h_i\}$ be the sequence of bad functions of $f$ as in Lemma
  \textup{\ref{lem-czdecom}} and $\big(\exp{\big(-\Delta^{(k)}_{\mu}t-\sigma t\big)}\big)_{t\ge 0}$
  the heat semigroup associated to $-(\Delta^{(k)}_{\mu}+\sz)$ with $\sz>\tilde{\sigma}$,
  where $\tilde{\sigma}$ is as in Lemma \ref{lem3}. Then there
  exists a constant $C>0$ independent of $f$ such that
  \begin{align*}
    \bigg\|\sum_i\exp{\Big(-t_i\Delta^{(k)}_{\mu}-\sigma t_i\Big)}\tilde{h}_i\bigg\|_2^2\leq C\lambda \,\|a\|_1
  \end{align*}
  where  $\tilde{h}_i= |h_i| \frac{a}{|a|}$ and $t_i=r_i^2$ with $r_i$ denoting the radii of the balls $B_i$ as in Lemma \textup{\ref{lem-czdecom}\,(b)}.
\end{lemma}

\begin{proof}
  Recall that $\supp h_i \subset B(x_i, \sqrt{t_i})$. Using the upper
  bound of the heat kernel in Lemma \ref{Gaussian-estimate} and
  Lemma \ref{lem-czdecom} (b), we have for $x\in M$,
  \begin{align*}
    \Big|\exp{\Big(-\Delta^{(k)}_{\mu}t_i-\sigma t_i\Big)}\tilde{h}_i(x)\Big|
    &\leq C\int_M \frac{\e^{-\sz' t_i-\alpha \frac{\rho^2(x,y)}{t_i}}}{\mu(B(x,\sqrt{t_i}))}\,|\tilde{h}_i(y)|\,\mu(\d  y)\\
    & \leq  \frac{C}{\mu(B(x,\sqrt{t_i}))} \,\e^{-\sz' t_i-\alpha\frac{\rho^2(x,x_i)}{t_i}}\int_{\tilde{B}_i} |h_i(y)|\,
      \mu(\d y)\\
    &\leq  C_1 \lambda \int_M \frac{\e^{-\sz' t_i-\alpha\frac{{ \rho^2(x,y)}}{t_i}}}{\mu(B(x,\sqrt{t_i}))}\1_{\tilde{B}_i}(y)\,\mu(\d y),
  \end{align*}
  for suitable $\sigma $ such that $\sz'=\sz-\tilde{\sz}>0$.
  It is therefore sufficient to verify that
  \begin{align}\label{lem-le1}
    \biggl\|\sum_i \int_M \frac{\e^{-\sz' t_i-\alpha'' \frac{\rho^2(\newdot, y)}{t_i}}}
    {\mu(B(\newdot, \sqrt{t_i}))}\1_{\tilde{B}_i}(y)\,\mu(\d y)\biggr\|_2\ls\biggl\|\sum_i\1_{\tilde{B}_i}\biggr\|_2,
  \end{align}
  since as consequence from this and Lemma \ref{lem-czdecom} we obtain  
  \begin{align*}
    \Big\|\sum_i\exp{(-\Delta^{(k)}_{\mu}t_i-\sigma t_i)}\tilde{h}_i\Big\|^2_2\ls\lambda^2\biggl\|\sum_i\1_{\tilde{B}_i}\biggr\|^2_2\ls \lambda^2 \sum_i \mu(\tilde{B}_i)\ls \lambda\, \|a\|_1.
  \end{align*}
  In order to prove \eqref{lem-le1}, we write by duality
  \begin{align}\label{lem-le2}
    &\biggl\|\sum_i \int_M \frac{\e^{-\sz' t_i-\alpha''\frac{\rho^2(\newdot,y)}{t_i}}}{\mu(B(\newdot, \sqrt{t_i}))}\1_{\tilde{B}_i}(y)\,\mu(\d y)\biggr\|_2\\ \notag
    &\qquad =\sup_{\|u\|_2=1}\lf|\int_M\lf(\sum_i\int_M\frac{\e^{-\sz' t_i-\alpha'' \frac{\rho^2(x,y)}{t_i}}}{\mu(B(x, \sqrt{t_i}))}\1_{\tilde{B}_i}(y)\,\mu(\d y)\r)u(x)\,\mu(\d x)\r|\\ \notag
    &\qquad
      \leq \sup_{\|u\|_2=1}\int_M \sum_i\lf(\int_M \frac{\e^{-\sz' t_i-\alpha''\frac{\rho^2(x,y)}{t_i}}}{\mu(B(x, \sqrt{t_i}))}|u(x)|\,\mu(\d x)\r)\1_{\tilde{B}_i}(y)\,\mu(\d y).
  \end{align}
  By the local doubling property \eqref{LD}, we have for any $x\in M$ and $y\in \tilde{B}_i$,
  \begin{align*}
   \mu(B(y, \sqrt{t_i}))\leq C\lf(1+\frac{\rho(x,y)}{\sqrt{t_i}}\r)^{m} \e^{\sigma_1\rho(x,y)/\sqrt{t_i}}
   \mu(B(x, \sqrt{t_i})).
  \end{align*}
  From this, we obtain that there exist
  $0<\tilde{\az}<\az'<\az$ such that
  \begin{align*}
    &\int_M \frac{\e^{-\sz' t_i-\alpha \frac{\rho^2(x,y)}{t_i}}}{\mu(B(x,\!\sqrt{t_i}))}\,|u(x)|\,\mu(\d x)\\
    &\hs\ls\frac{\e^{-\frac{1}{2}\sz' t_i}}{\mu(B(y,\!\sqrt{t_i}))}\int_M \e^{-\alpha'\frac{\rho^2(x,y)}{t_i}}|u(x)|\,\mu(\d x)\\
    &\hs\ls\frac{1}{\mu(B(y,\!\sqrt{t_i}))}\Bigg(\int_{\rho(x,y)<\sqrt{t_i}}|u(x)|\,\mu(\d x)
      +\sum_{k=0}^\fz\int_{2^k\!\sqrt{t_i}\leq \rho(x,y)< 2^{k+1}\!\sqrt{t_i}}\e^{-\alpha'\frac{\rho^2(x,y)}{t_i}}|u(x)|\,\mu(\d x)\Bigg)\\
    &\hs\leq \frac{1}{\mu(B(y,\!\sqrt{t_i}))}\lf(\int_{B(y,\!\sqrt{t_i})}|u(x)|\,\mu(\d x)+\sum_{k=0}^\fz\e^{-\alpha' 2^{2k}}\int_{B(y, 2^{k+1}\sqrt{t_i})}|u(x)|\,\mu(\d x)\r)\\
    &\hs=\lf(1+\sum_{k=0}^\fz\frac{\mu(B(y,\,\!2^{k+1}\sqrt{t_i}))}{\mu(B(y,\! \sqrt{t_i}))}\,\e^{-\alpha'2^{2k}}\r)(\SM u)(y)\\
    &\hs\leq \left(1+C\sum_{k=0}^\fz2^{(k+1)m}\e^{ \sigma_1(2^{k+1}-1)\sqrt{t_i}} \e^{-\tilde{\az} 2^{2k}}\right)(\SM u)(y)\\
    & \hs\leq \left(1+C\sum_{k=0}^\fz2^{(k+1)m}\e^{c2^k-\tilde{\az} 2^{2k}}\right)(\SM u)(y)\ls (\SM u)(y),
  \end{align*}
  where \begin{align*} (\SM
    u)(y):=\sup_{r>0}\frac{1}{\mu(B(y,r))}\int_{B(y,r)}|u(x)|\,\mu(\d x)
        \end{align*}
        denotes the Hardy-Littlewood maximal function of $u$.  This
        together with \eqref{lem-le2} and the $L^2$-bounded\-ness of
        $\SM $ gives
        \begin{align*}
          \biggl\|\sum_i\int_M \frac{\e^{-\alpha''\frac{\rho^2(\newdot, y)}{t_i}}}{\mu(B(\newdot,\!\sqrt{t_i}))}\1_{\tilde{B}_i}(y)\,\mu(\d y)\biggr\|_2\ls
          \sup_{\|u\|_2=1}\int_M (\SM u)(y)\sum_{i}\1_{\tilde{B}_i}(y)\,\mu(\d y)
          \ls\biggl\|\sum_i\1_{\tilde{B}_i}\biggr\|_2,
        \end{align*}
        which shows that \eqref{lem-le1} holds true and finishes the
        proof of Lemma \ref{lem-le}.
      \end{proof}

      With the help of the Lemmata \ref{lem-bcT} through \ref{lem-le}, we
      are now in position to give a proof of Theorem \ref{main-them1}. Note that
      Theorem \ref{them2} can be established along the same lines, with the slight
      difference that in this case $\sigma$ can be taken to be $0$.

\begin{proof}[Proof of Theorem \ref{main-them1}] 
  Recall that $T_{\sigma}^{(k)}=\nabla(\Delta^{(k)}_{\mu}+\sz)^{-1/2}$. We choose $\sigma$ large
  enough such that $\sigma>2C_0$ where $C_0$ 
  is defined as in Theorem \ref{L2-estimate}.  By Lemma \ref{lem-bcT}, it
  suffices to prove
  \begin{align}\label{eqn-localT-1}
    \mu\big(\{x\in 2B\colon  |T_{\sigma}^{(k)}a (x)|>\lambda\}\big)\ls \dfrac{\|a\|_1 }{\lambda},\quad \lambda\in (0,\infty)
  \end{align}
  for all $a\in \Gamma_{C_0^\fz}(\Lambda ^{k}T^*M)$. By means of Lemma \ref{lem-czdecom} with
  $\mathcal{X}=B$, we deduce that $f$ has a decomposition
  $$\textstyle |a|=g+h=g+\sum_ih_i$$ which implies
  \begin{align}\label{eqn-them-8}
   & \mu\big(\{x\in 2B: |T_{\sigma}^{(k)}a(x)|>\lambda\}\big) \notag \\
    &\leq \mu\lf(\lf\{x\in 2B:|T_{\sigma}^{(k)}\tilde{g}(x)|>\frac{\lambda}{2}\r\}\r)+
      \mu\lf(\lf\{x\in 2B: |T_{\sigma}^{(k)}\tilde{h}(x)|>\frac{\lambda}{2}\r\}\r)\\ \notag
    &=:\mathrm{I}_1+\mathrm{I}_2,
  \end{align}
  where $$\tilde{g}=g\frac{a}{|a|}\quad \mbox{and}\quad \tilde{h}=h\frac{a}{|a|}.$$ 
  Using the facts that $T_{\sigma}^{(k)}$ is bounded on $L^2(\mu)$ and that
  $|g(x)|\leq C\lambda$, we obtain as desired
  \begin{align}\label{eqn-them-9}
    \mathrm{I}_1\ls\lambda^{-2}\|T_{\sigma}^{(k)}\tilde{g}\|_2^2\ls\lambda^{-2}\|\,| \tilde{g}|\,\|_2^2
    \ls\lambda^{-1}\|g\|_1\ls\lambda^{-1}\|a\|_1.
  \end{align}

  We now turn to the estimate of $\mathrm{I}_2$. Recall that
  $\exp{(-\Delta^{(k)}_{\mu}t-\sigma t)},\, t\geq 0$ is the heat semigroup generated by
  $-(\Delta^{(k)}_{\mu}+\sz)$.  We write
  \begin{align*}
    T_{\sigma}^{(k)}\tilde{h}_i=T_{\sigma}^{(k)}\exp{(-\Delta^{(k)}_{\mu}t_i-\sigma t_i)}\tilde{h}_i+T_{\sigma}^{(k)}\Big(I-\exp{(-\Delta^{(k)}_{\mu}t_i-\sigma t_i)}\Big)\tilde{h}_i,
  \end{align*}
  where $t_i=r_i^2$ with $r_i$ the radius of $\tilde{B}_i$.  By Lemma
  \ref{lem-le}, we have
  \begin{align*}
    \biggl\|\sum_i\exp{(-\Delta^{(k)}_{\mu}t_i-\sigma t_i)}\tilde{h}_i\biggr\|_2^2\ls\lambda\, \|a\|_1.
  \end{align*}
  This combined with the $L^2$-boundedness of $T_{\sigma}^{(k)}$ yields
  \begin{align}\label{eqn-them1-1}
    \mu\left(\left\{x\in 2B\colon \lf|T_{\sigma}^{(k)}\lf(\sum_i \exp{(-\Delta^{(k)}_{\mu}t_i-\sigma t_i)}\tilde{h}_i\r)(x)\r|>\frac{\lambda}{2}\right\}\right)
    \ls \frac{1}{\lambda}\|a\|_1
  \end{align}
  as desired. Consider now the term $T_{\sigma}^{(k)}\sum_i(I-\exp{(-\Delta^{(k)}_{\mu}t_i-\sigma t_i)})\tilde{h}_i$. We
  write
  \begin{align}\label{eqn-them-4}
    &\mu\lf(\lf\{x\in 2B\colon\lf|T_{\sigma}^{(k)}\lf(\sum_i(I-\exp{(-\Delta^{(k)}_{\mu}t_i-\sigma t_i)})\tilde{h}_i\r)(x)\r|>\frac{\lambda}2\r\}\r)\\ \notag
    &\hs\leq \sum_{i}\mu(2\tilde{B}_i)+\mu\lf(\lf\{x\in 2B\setminus\cup_i2\tilde{B}_i\colon \lf|T_{\sigma}^{(k)}\lf(\sum_i(I-\exp{(-\Delta^{(k)}_{\mu}t_i-\sigma t_i)})\tilde{h}_i\r)\r|(x)>\frac{\lambda}{2}\r\}\r).
  \end{align}
  From Lemma \ref{lem-czdecom} we conclude that
  \begin{align}\label{eqn-them1-2}
    \sum_{i}\mu(2\tilde{B}_i)\ls\frac{\|a\|_1}{\lambda}.
  \end{align}
To estimate the second term, denote 
  the integral kernel of the operator $T_{\sigma}^{(k)}(I-\exp{(-\Delta^{(k)}_{\mu}t_i-\sigma t_i)})$ by $k_{t_i}^{\sz,k}(x,y)$.
  Note that
  \begin{align*}
    &(\Delta^{(k)}_{\mu}+\sz)^{-1/2}\big(I-\exp{(-\Delta^{(k)}_{\mu}t_i-\sigma t_i)}\big)\\
    &=\int_0^{\infty}\left(\frac{\exp{(-\Delta^{(k)}_{\mu}s-\sigma s)}}{\sqrt{s}}-\frac{\exp{(-\Delta^{(k)}_{\mu}(t_i+s)-\sigma (t_i+s))}}{\sqrt{s}}\right)\,\d s\\
    &=\int_0^{\infty}\left(\frac{1}{\sqrt{s}}-\frac{\1_{\{s\geq t_i\}}}{\sqrt{s-t_i}}\right)\exp{(-\Delta^{(k)}_{\mu}s-\sigma s)}\,\d s
  \end{align*}
  and
  \begin{align*}
   & T_{\sigma}^{(k)}(I-\exp{(-\Delta^{(k)}_{\mu}t_i-\sigma t_i)})\\
   &\qquad=\nabla(\Delta^{(k)}_{\mu}+\sigma)^{-1/2}(I-\exp{(-\Delta^{(k)}_{\mu}t_i-\sigma t_i)})\\
   &\qquad=\int_0^{\infty}\left(\frac{1}{\sqrt{s}}-\frac{\1_{\{s\geq t_i\}}}{\sqrt{s-t_i}}\right)\e^{-\sigma s}\nabla \exp{(-s\Delta^{(k)}_{\mu})}\,\d s.
  \end{align*}
  Therefore,
  \begin{align}\label{eqn-kernelKt}
    k_{t_i}^{\sz,k}(x,y)=\int_0^{\infty}\e^{-\sigma s}\left(\frac{1}{\sqrt{s}}-\frac{\1_{\{s\geq t_i\}}}{\sqrt{s-t_i}}\right)\nabla_x\exp{\big(-s\Delta^{(k)}_{\mu}\big)}(x,y)\,\d s.
  \end{align}
   Since $\tilde{h}_i$ is supported in $\tilde{B}_i$, we have
  \begin{align}\label{eqn-them-3}
   & \int_{2B\setminus (2\tilde{B}_i)}
    \lf|T_{\sigma}^{(k)}\lf((I-\exp{(-\Delta^{(k)}_{\mu}t_i-\sigma t_i)})\tilde{h}_i\r)(x)\r|\,\mu(\d x) \notag\\
    &\qquad \leq 
      \int_{2B\setminus (2\tilde{B}_i)}\lf(\int_{\tilde{B}_i}|k_{t_i}^{\sz,k}(x,y)|\,|h_i(y)|\,\mu(\d y)\r)\,\mu(\d x) \notag\\ 
    &\qquad \leq \int_{\tilde{B}_i}\lf(\int_{\rho(x,y)\geq t_i^{1/2}}|k_{t_i}^{\sz,k}(x,y)|\,\mu(\d x)\r)|h_i(y)|\,\mu(\d y).
  \end{align}
  Now by means of \eqref{eqn-kernelKt} and Theorem \ref{cor-1}, we
  get
  \begin{align*}
    \int_{\rho(x,y)\geq t_i^{1/2}}|k_{t_i}^{\sz, k}(x,y)|\,\mu(\d x)
    &\leq 
      \int_0^{\infty}\lf(\int_{\rho(x,y)\geq t_i^{1/2}}|\nabla_x \exp{(-\Delta_{\mu}^{(k)}s)}(x,y)|\,\mu(\d x)\r)\e^{-\sigma s}\Big|\frac{1}{\sqrt{s}}-\frac{\1_{\{s\geq t_i\}}}{\sqrt{s-t_i}}\Big | \,\d s\\
    &\leq C\int_0^{\infty} \e^{-\gamma t_i/2s}\,\e^{2C_0 s}\frac{(1+\sqrt{\sigma_3 s})}{\sqrt{s}}\Big|\frac{1}{\sqrt{s}}-\frac{\1_{\{s\geq t_i\}}}{\sqrt{s-t_i}}\Big |\,\e^{-\sigma s}\,\d s\\
    &\leq C\int_0^{\infty}\e^{-\gamma/2u}\left|\frac{1}{u}-\frac{\1_{\{u\geq 1\}}}{\sqrt{u(u-1)}}\right|\,\d u\\
    &=C \int_0^1\frac{\e^{-\gamma/2u}}{u}\,\d u +C\int_1^{\infty}\left(\frac{1}{\sqrt{u(u-1)}}-\frac{1}{u}\right)\,\d u<\infty
  \end{align*}
  where for the third line above we used the fact that
  $$\e^{s(2C_0-\sigma) }(1+\sqrt{\sigma_3 s}),\quad s\in (0,\fz),$$
  is bounded.
The estimate above together with \eqref{eqn-them-3} and Lemma
\ref{lem-czdecom} implies that
\begin{align}\label{eqn=them-6}
  \mu\lf(\lf\{x\in 2B\setminus\cup_i2\tilde{B}_i\colon
  \lf|T_{\sigma}^{(k)}\lf(\sum_i\big(I-\exp{(-\Delta^{(k)}_{\mu}t_i-\sigma t_i)}\big)\tilde{h}_i\r)(x)\r|>\frac{\lambda}{2}\r\}\r)
  \ls \frac{\|a\|_1}{\lz}.
\end{align}
Altogether, combining \eqref{eqn-them-8} through \eqref{eqn-them1-1},
\eqref{eqn-them1-2} and \eqref{eqn=them-6}, we conclude that
\eqref{eqn-localT-1} holds which completes the proof of Theorem
\ref{them2}.
\end{proof}

\section{Proof of Theorem \ref{them-CZ}}\label{s3}

By the Bishop-Gromov comparison theorem and the well-known formula for
the volume of balls in the hyperbolic space,  the local volume doubling property \eqref{LD} holds if
the curvature-dimension condition \eqref{eqn-CD} is satisfied.

\begin{proof}[Proof of Theorem \ref{them-CZ}]
Let us sketch the main idea of the second method in \cite{GP-15}.
Inequality \ref{CZ} is reduced to the existence of
positive constants $C$ and $\sigma$ such that
\begin{align*}
  \left\|\,|\Hess (\Delta_{\mu}+\sigma)^{-1}u|\,\right\|_p\leq C\|u\|_p,
\end{align*}
which is equivalent to
\begin{align*}
  \left\|\,|\nabla (\Delta_{\mu}^{(1)}+\sigma)^{-1/2}
  \circ \d(\Delta_{\mu}+\sigma)^{-1/2}u|\,\right\|_p\leq C\|u\|_p.
\end{align*}
The problem is thus reduced to the study of conditions for boundedness
of the classical Riesz transform $\d(\Delta_\mu+\sigma)^{-1/2}$ on
functions and boundedness of the covariant Riesz transform
$\nabla (\Delta^{(1)}_\mu+\sigma)^{-1/2}$ on one-forms. 

As far as the covariant Riesz transform
$\nabla (\Delta_{\mu}^{(1)}+\sigma)^{-1/2}$ on one-forms is concerned, this transform is bounded in $L^p(\mu)$ for $1<p\leq 2$  by Theorem \ref{them2}.
If the 
local volume doubling property and short time Gaussian estimate for the heat kernel hold, then
boundedness of the classical Riesz transform $\d(\Delta_{\mu}+\sigma)^{-1/2}$ holds for $1<p\leq 2$ as well.
Note that \eqref{eqn-CD} for $K_0\in \R$ implies \eqref{eqn-Ric} for $K_0$, and the curvature condition \eqref{eqn-Ric} assures the short time Gaussian estimate for the heat kernel (see \cite[Theorem 2.4.4]{Wbook14}). \end{proof}

\bibliographystyle{amsplain}
\bibliography{Covariant-RT}

\end{document}